\documentclass[11pt,reqno,draft]{amsart}
\usepackage[margin=1in]{geometry}
\usepackage[usenames]{color}
\usepackage{amsmath,pdfsync,verbatim,graphicx,epstopdf,enumerate}
\usepackage[colorlinks=true]{hyperref}
\usepackage{cancel}
\usepackage[framemethod=tikz]{mdframed}
\hypersetup{allcolors=black}
\numberwithin{equation}{section}

\newcommand{\I}{\mathrm{i}}

\newcommand{\wh}{\widehat}

\newcommand{\lb}{\left(}

\newcommand{\rb}{\right)}
\newcommand{\PD}{\partial}
\renewcommand{\d}{\delta}
\newcommand{\Beq}{\begin{equation}}
\newcommand{\Eeq}{\end{equation}}
\newcommand{\beq}{\begin{equation*}}
\newcommand{\eeq}{\end{equation*}}
\newcommand{\bal}{\begin{align}}
\newcommand{\eal}{\end{align}}
\renewcommand{\O}{\Omega}

\newcommand{\g}{\gamma}

\newcommand{\n}{\nabla}

\usepackage{mathtools}

\newcommand{\B}{\beta}
\newcommand{\bp}{\begin{prob}}
	\newcommand{\ep}{\end{prob}}
\newcommand{\bpr}{\begin{proof}}
	\newcommand{\epr}{\end{proof}}
\renewcommand{\o}{\omega}

\newcommand{\st}{\,:\,}

\newcommand{\bel}[1]{\begin{equation}\label{#1}}
\newcommand{\ee}{\end{equation}}

\newcommand{\ssubset}{\subset\joinrel\subset}

\newtheorem{theorem}{Theorem}[section]
\newtheorem{corollary}[theorem]{Corollary}

\newtheorem{lemma}[theorem]{Lemma}
\newtheorem{proposition}[theorem]{Proposition}

\theoremstyle{definition}
\newtheorem{definition}[theorem]{Definition}

\newtheorem{remark}[theorem]{Remark}

\newcommand{\Rn}{\mathbb{R}^n}
\newcommand{\R}{\mathbb{R}}
\newcommand{\D}{\mathrm{d}}
\newcommand{\Lc}{\mathcal{L}}
\newcommand{\Rb}{\mathbb{R}}
\newcommand{\Dc}{\mathcal{D}}
\newcommand{\Nc}{\mathcal{N}}
\newcommand{\A}{\alpha}
\newcommand{\vp}{\varphi}
\newcommand{\Oc}{\mathcal{O}}
\newcommand{\Cb}{\mathbb{C}}
\newcommand{\Lm}{\left\lvert}
\newcommand{\Rm}{\right\rvert}
\newcommand{\ve}{\varepsilon}

\newcommand{\mh}{h_1}

\newcommand{\Wc}{\mathcal{W}}

\usepackage[usenames]{color}
\usepackage{amsmath,pdfsync,verbatim,graphicx,epstopdf,enumerate}
\renewcommand{\d}{\delta}
\newcommand{\wt}{\widetilde}

\newcommand{\Bc}{\mathcal{B}}

\newcommand{\Ec}{\mathcal{E}}

\newcommand{\Sb}{\mathbb{S}}
\newcommand{\Sn}{\mathbb{S}^{n-1}}

\newcommand{\Sm}{\mathbf{S}}

\newcommand{\aK}{\mathfrak{a}}
\newcommand{\bK}{\mathfrak{b}}

\renewcommand{\O}{\Omega}

\renewcommand{\o}{\omega}

\newcommand{\sbc}[1]{\begin{quotation}\textbf{\color{teal}Sombuddha's comment:\
}{\color{teal}\textit{#1}}\end{quotation}}

\newcommand{\vkc}[1]{\begin{quotation}\textbf{\color{red}Venky's comment:\
}{\color{red}\textit{#1}}\end{quotation}}

\title[Inverse problem for a polyharmonic operator]{Unique determination of 
anisotropic perturbations of a polyharmonic operator from partial boundary data}
\author[Bhattacharyya, Krishnan and Sahoo]{Sombuddha Bhattacharyya$^\ast$, Venkateswaran P. Krishnan$^{\mathsection}$ and Suman Kumar Sahoo$^{\dagger}$}
\address {$^{\ast}$ Department of Mathematics, Indian Institute of Science Education and Research, Bhopal.
\newline
E-mail:{\tt\  sombuddha@iiserb.ac.in}}
\address{$^{\mathsection}$ Centre for Applicable Mathematics, Tata Institute of Fundamental Research, India.
\newline
E-mail:{\tt \ vkrishnan@tifrbng.res.in}}
\address{$^{\dagger}$ Department of Mathematics and Statistics, University of Jyv\"askyl\"a, Finland.
\newline
E-mail:{\tt \ suman.k.sahoo@jyu.fi}}
\begin{document}
	
\begin{abstract}
	We study an inverse problem involving the unique recovery of several lower order anisotropic tensor perturbations of a polyharmonic operator in a bounded domain from  the knowledge of the Dirichlet to Neumann map on a part of boundary. The uniqueness proof relies on the inversion of generalized momentum ray transforms (MRT) for symmetric tensor fields, which we introduce for the first time to study Calder\'on-type inverse problems. We construct suitable complex geometric optics (CGO) solutions for the polyharmonic operators that reduces the inverse problem to uniqueness results for a generalized MRT. The uniqueness result and the inversion formula we prove for generalized MRT could be of independent interest and we expect it to be applicable to other inverse problems for higher order operators involving tensor perturbations.  
	
\end{abstract}

	
	\maketitle
	
	\section{Introduction and statement of the main result}
	
	Let $\O\subset \Rb^n$ for $n\geq 3$ be a bounded domain with smooth boundary $\PD\O$.  For $m\geq 2$, we consider the following polyharmonic operator $\Lc$ with lower order anisotropic/tensor perturbations  up to order $m$: 

\begin{equation}\label{operator}
\Lc(x,D)=(-\Delta)^m + \sum\limits_{j=0}^{m-1} \sum\limits_{i_1,\cdots, i_{m-j}=1}^{n}A_{i_{1}\cdots i_{m-j}}^{m-j}(x) D^{m-j}_{i_{1}\cdots i_{m-j}}+ q(x).
\end{equation}
Here as usual, $D^{k}_{i_{1}\cdots i_{k}}=\frac{1}{\I^{k}}\frac{\PD^{k}}{\PD x_{i_1}\cdots \PD x_{i_{k}}}$. 
We assume that $q(x)\in L^{\infty}(\O)$ and all the other coefficients $\{A^{j}\}$ for $1\leq j\leq m$  are symmetric in its indices and belong to $C^{\infty}(\overline{\O})$. To avoid proliferation of indices, we denote the symmetric tensor field $(A^{j}_{i_1\cdots i_j})$ as simply $(A^j)$.

	Let the domain of the operator $\Lc(x,D)$ be 
	\begin{equation}\label{domain}
		\Dc(\Lc(x,D)) = \Big{\{}u\in H^{2m}(\Omega): u|_{\partial\Omega}=(-\Delta)u|_{\partial\Omega}=\cdots=(-\Delta)^{m-1}u|_{\partial\Omega} =0\Big{\}}.
	\end{equation}
	
	We make the assumption that $0$ is not an eigenvalue of $\Lc(x,D): \Dc(\Lc(x,D)) \to L^2(\O)$ for the set of coefficients $(A^{j})$ and $q$ under consideration. We denote 
	\begin{equation}\label{trace_map}
		\gamma u = \Big{(} u |_{\partial\Omega},\cdots,(-\Delta)^k u |_{\partial\Omega},\cdots, (-\Delta)^{m-1}u |_{\partial\Omega}\Big{)} \in \prod_{k=0}^{m-1} H^{2m-2k-\frac{1}{2}}(\partial\Omega).
	\end{equation}
	Then for any $ f = (f_0,f_1,..,f_{m-1}) \in \prod_{k=0}^{m-1} H^{2m-2k-\frac{1}{2}}(\partial\Omega)$, the boundary value problem, 
	\begin{equation}\begin{aligned}\label{problem}
			\Lc(x,D)u &= 0\quad \mbox{ in }\Omega, \\
			\gamma u&= f \quad \mbox{ on }\partial\Omega,
		\end{aligned}
	\end{equation}
	has a unique solution $u_f\in H^{2m}(\Omega)$.
	
	We define the Neumann trace operator $\gamma^{\#} $ by
	\[
	\gamma^{\#} u = \Big{(}\partial_{\nu} u |_{\partial\Omega},\cdots,\partial_{\nu}(-\Delta)^k u |_{\partial\Omega},\cdots,\partial_{\nu} (-\Delta)^{m-1}u |_{\partial\Omega}\Big{)},\\
	\]
	where $\nu$ is the outward unit normal to the boundary $\partial\Omega$. 	Let $\Gamma \subset \partial\Omega$ be a non-empty open subset of the boundary. We define the partial Dirichlet to Neumann (DN) map as
	\begin{equation}\label{partial_Neumann}
		\mathcal{N}_{\Gamma}(f) := \gamma^{\#}u_f|_{\Gamma}
		\in \prod_{k=0}^{m-1} H^{2m-2k-\frac{3}{2}}(\Gamma),
	\end{equation}

Let $\vp(x) = x_1$ and  
$\nu$ be the outward unit normal vector field to $\PD \O$. 
Define the back and front faces $\PD_{\pm} \O$ of $\PD \O$ as follows:  
\begin{equation}\label{parts_of_boundary}
    \partial_{\pm}\Omega = \{ x\in \partial\Omega \st \pm \PD_{\nu}\vp(x)= \pm \nu_1 \geq 0\}.
\end{equation}
Let $\Gamma_F, \Gamma_{B} \subset \partial \Omega$ be open neighborhoods of $\partial_{-}\Omega$  and $\partial_{+}\Omega$ respectively.

Before we state the main result, let us mention our assumption on the leading order tensor field $A^m$. We assume that this tensor field  has a partial isotropy in the following sense. There exists a smooth symmetric $m-2$ tensor field $\overline{A}^{\:m-2}$  in $\O$ such that $A^m= i_{\delta} \overline{A}^{\:m-2} =\overline{A}^{\:m-2}\odot \delta$, where $\delta$ is the Kronecker delta tensor and $\odot$ denotes symmetrization of the indices. This constraint on the leading coefficient tensor field is required for the techniques employed in this paper to work. See Remarks \ref{Rem_isotropy_1} and \ref{Rem_isotropy_2} for more details.

	The main result of the article is that the partial DN map uniquely determines \emph{all} the coefficients $\lb A^{j}\rb$ and $q$. 
	\begin{theorem}\label{mainresult_1}
		Let $\Omega\subset \R^n$ for $n\geq 3$ be a bounded domain with smooth boundary and  $\Gamma_F\subset \partial\Omega$ be an open neighborhood of $\PD_{-}\O$ defined in \eqref{parts_of_boundary}. Let us consider two operators $\Lc(x,D)$ and $\wt{\Lc}(x,D)$ with corresponding coefficients being denoted by $(A^{j})_{1\leq j\leq m}$, $q$ and $(\wt{A}^{j})_{1\leq j\leq m}$, $\wt{q}$, respectively. 
		If $\Nc_{\Gamma_F}$ and $\wt{\Nc}_{\Gamma_{F}}$ denote the partial Dirichlet to Neumann maps corresponding to $\Lc$ and $\wt{\Lc}$, respectively, then  $\Nc_{\Gamma_F}=\wt{\Nc}_{\Gamma_{F}}$ implies $(A^{j}) = (\wt{A}^{j})$ for all $1\leq j\leq m$ and $q=\wt{q}$ on $\overline{\O}$.
	\end{theorem}



The new ingredient in our proof is the inversion of generalized momentum ray transforms (MRT) of symmetric tensor fields. MRT, for theoretical reasons, was introduced by Sharafutdinov in \cite{Sharafutdinov_1986_momentum,Sharafutdinov_book}.  Recently this transform was studied in  greater detail in \cite{KMSS,KMSS_range,Rohit_Suman}.  In our work, in fact, we study a generalization of MRT, in the sense that the transform is a sum of MRT of symmetric tensor fields up to a certain order. Furthermore, the added difficulty in our set-up is that the transform is only known on the unit sphere bundle of $\Rb^n$. While in the usual ray transform or even in the standard  MRT studied earlier, the knowledge of these transforms on the unit sphere bundle or on $\Rb^n\times \Rb^n\setminus \{0\}$ are equivalent, here it is not the case due to the fact that we have a sum of MRT of different rank symmetric tensor fields, and each one has different degree of homogeneity in the direction variable. Furthermore, due to the fact that, in our problem, the generalized MRT is only known on the unit sphere bundle, this transform has a non-trivial kernel. We require a further iterative argument to show the unique recovery of the tensor fields from generalized MRT in our inverse problem. 

Until now, to the best of our knowledge, no direct or indirect application motivating the study of momentum ray transforms of symmetric tensor fields has existed in literature. Our current work is the first known study where an application of the inversion of MRT of symmetric tensor fields is used in the solution of inverse problems. The authors expect the inversion of such MRT to arise in several other situations as well, especially, in inverse  problems for higher order operators involving semilinear and power-type non-linearities.
	
	Another advantage of the analysis of MRT in the context of our inverse problem is that we can relax the restriction that the domain be simply connected as was required in the works \cite{KRU2,KRU1,Ghosh-Krishnan,BG_19,BG20}. This assumption is required in these works since the recovery of a vector field perturbation (for instance) from boundary Dirichlet-to-Neumann data was done in two steps. First, the vector field term is shown to be potential and then one shows that the the potential term is $0$; it is at this stage that the domain being simply connected is required. However, using our approach, we are able to directly show   the unique recovery of tensor perturbations from boundary data.
	
We now give a brief history of inverse problems for PDEs, in relation to the problem we study. 
The study in this direction was initiated by Calder\'on in his fundamental article \cite{Calderon_Paper}. This was soon followed by several deep and insightful works in this area; see \cite{SYL,Nachman_annals,AP,KEN} to name a few.
In recent years, inverse problems involving higher order operators, motivated by potential applications in fields such as elasticity and conformal geometry, have attracted significant attention. These problems address the recovery of several lower order coefficients of the operator based on boundary Dirichlet to Neumann measurements. 
Another related set of problems that has been investigated well is inverse problems involving system of PDEs. We refer to the following works in this area \cite{Somarsalo_Isaacson_Cheney_inv_maxwell,Salo-Tzou, COS09, Kenig_Salo_Uhlmann_duke, Caro_Zhou_time_harmonic_maxwell,CST_Hodge_Laplace_inv}.

In \cite{KRU2} the authors showed uniqueness of the zeroth order and the first order vector field perturbations of a biharmonic operator from the Dirichlet to Neumann map measured only on a part of the boundary. The work \cite{KRU1} established the uniqueness of the zeroth order and first order perturbations of a polyharmonic operator from Dirichlet to Neumann map measured on the whole boundary. Later \cite{Ghosh-Krishnan} proved unique determination of multiple isotropic perturbations, given by either a function or a vector field, of a polyharmonic operator from partial boundary measurements.
Recently \cite{BG_19,BG20} proved unique recovery of a symmetric matrix appearing as a second order perturbation of a polyharmonic operator along with the first and the zeroth order perturbations from boundary measurements. We should mention that there are very few works in the field of inverse boundary problems addressing recovery of anisotropic perturbations from boundary data and our work is a step in this direction. We  substantially generalize the existing results on inverse problems for polyharmonic operators by showing unique recovery of symmetric $k$ tensor fields appearing as perturbations of a polyharmonic operator  for $0 \leq k \leq m$, from partial boundary measurements. In the process, as mentioned already, we believe the techniques introduced here could potentially be applied to other inverse problems involving anisotropic perturbations.


This paper is organized as follows. In Section \ref{Sec_Prelim}, we give the preliminary results required in the proof of our main results -- interior and boundary Carleman estimates, CGO solutions and suitable solutions of transport equations. 
	Section \ref{determination} is devoted to the uniqueness proof assuming the inversion of generalized MRT which we prove in Section \ref{Sec_MRT}. Appendix \ref{Appendix_elliptic_regu} gives a regularity estimate well-known in the classical setting but we could not find a proof of this estimate in the semi-classical setting in the literature, which we require for our boundary Carleman estimate in Section \ref{Sec_Prelim}.

\section{Preliminaries}\label{Sec_Prelim}
In this section, we collect all the preliminary results required for the proof of our main result. Construction of special solutions to the transport equation (see Lemma \ref{Prop4.1}) and the boundary Carleman estimate (see Proposition \ref{bdy_Car_Est}) are the new results in this section. Therefore only the proofs of these two results are given.

\subsection{Interior Carleman estimate}\label{Carleman Estimates}
Let $\Omega\subset \R^n$, $n\geq 3$ be a bounded domain with smooth boundary. In this section, we construct Complex Geometric Optics (CGO) solutions for \eqref{operator} following a Carleman estimate approach from \cite{BUK,DOS}. 

We first introduce semiclassical Sobolev spaces. 
For $h>0$ be a small parameter, we define the semiclassical Sobolev space $H^s_{\mathrm{scl}}(\mathbb{R}^n)$, $s \in \R$ as the space $H^s(\Rn)$ endowed with the semiclassical norm
\Beq	\lVert u \rVert^2_{H^s_{\mathrm{scl}}(\Rn)} = \lVert \langle hD\rangle^s \, u \rVert^2_{L^2(\Rn)}, \quad \langle \xi \rangle= (1+|\xi|^2)^{\frac{1}{2}}.
\label{Eq:2.1}
\Eeq
For open sets $\Omega \subset\Rn$ and for non-negative integers $m$, the semiclassical Sobolev space $H^m_{\mathrm{scl}}(\Omega)$ is the space $H^m(\Omega)$ endowed with the following semiclassical norm
\Beq\label{Eq:2.2}
\lVert u \rVert^2_{H^m_{\mathrm{scl}}(\Omega)} = \sum\limits_{|\alpha|\le m} \lVert ( hD)^{\alpha}\, u \rVert^2_{L^2(\Omega)}. \Eeq
For an integer $m\geq 0$, \eqref{Eq:2.1} and \eqref{Eq:2.2} give equivalent norms when $\O=\Rb^n$.

Let $\wt{\O}$ be an open subset containing $\O$ in its interior and let $\vp\in C^{\infty}(\wt{\O})$ with $\n \vp\neq 0$ in $\overline{\O}$. We consider the semi-classical conjugated Laplacian $P_{0,\vp} = e^{\frac{\vp}{h}}\lb -h^2 \Delta\rb e^{\frac{-\vp}{h}}$, where $0<h\ll 1$. The semi-classical principal symbol of this operator $P_{0,\vp}$ is given by $p_{0,\vp}(x,\xi)=|\xi|^2 - |\nabla_x\vp|^2 + 2\I \xi\cdot\nabla_x\varphi$.

\begin{definition}[\cite{KEN}] \label{LCW}
	We say that $\vp\in C^{\infty}(\wt{\O})$ is a limiting Carleman weight for $P_{0,\vp}$
	in $\O$ if
	$\nabla\vp \neq 0$ in $\overline{\O}$ and $\mathrm{\mathrm{Re}}(p_{0,\vp})$, $\mathrm{Im}(p_{0,\vp})$ satisfies
	\[\Big{\{}\mathrm{Re}(p_{0,\vp}), \mathrm{Im}(p_{0,\vp})\Big{\}}(x,\xi)=0 \mbox{ whenever } p_{0,\vp}(x,\xi)=0 \mbox{ for } (x,\xi)\in \widetilde{\O}\times(\mathbb{R}^n\setminus\{0\}),
	\]
	where $\{\cdot,\cdot\}$ denotes the Poisson bracket.
\end{definition}
Examples of such $\vp$ are linear weights
$\vp(x) = \A\cdot x$, where $0\neq \A\in\mathbb{R}^n$ or logarithmic weights $\vp(x)= \log|x-x_0|$ with $x_0 \notin \overline{\widetilde{\O}}$.

As mentioned already, we consider the limiting Carleman weight in this paper to be $\vp(x)= x_1$.
Note that $-\vp$ is also a limiting Carleman weight for $P_{0,\vp}$ whenever $\vp$ is.
We only require this specific limiting Carleman weight in Section \ref{Section_coeff_deter} where we prove the unique determination of coefficients using MRT. For the interior and boundary Carleman estimates, any limiting Carleman weight would work and for this reason, we consider any $\vp$ satisfying Definition \ref{LCW} until Section \ref{Section_coeff_deter}. 

\begin{proposition}[Interior Carleman estimate]\label{Prop: Interior Carleman Estimate}
	Let us consider the operator $\Lc(x,D)$ as in \eqref{operator} with $\lb A^j\rb$ in $C^{\infty}(\overline{\Omega})$ and $q \in L^{\infty}(\Omega)$ and $\vp(x)$ be a limiting Carleman weight for the conjugated semiclassical Laplacian. Then there exists a constant $C=C_{\O, A^j, q}$ such that for $0< h\ll 1 $, we have
	\begin{equation}\label{Car_0}
		h^{m}\lVert u\rVert_{L^2(\Omega)} \leq C \lVert h^{2m}e^{\frac{\vp}{h}}\Lc(x,D)e^{-\frac{\vp}{h}}u\rVert_{H^{-2m}_{\mathrm{scl}}},\quad \mbox{ for all } u\in C^{\infty}_0(\Omega).
	\end{equation}
\end{proposition}
This follows by iterating the following Carleman estimate for the semiclassical Laplacian with a gain of two derivatives proved in \cite{Salo-Tzou}, and then absorbing the lower order terms. We omit the proof here.  We refer the reader to \cite{KRU1,KRU2,Ghosh-Krishnan,BG_19}.

	The Carleman estimate in Proposition \ref{Prop: Interior Carleman Estimate} is valid for the operator $\Lc^*(x,D)$ as well where $\Lc^*$ is the formal $L^2$ dual of the operator $\Lc$. This follows from the fact that the operators $e^{\frac{\vp}{h}}\Lc(x,hD)e^{-\frac{\vp}{h}}$ and $\left(e^{\frac{\vp}{h}}\Lc(x,hD)e^{-\frac{\vp}{h}}\right)^* = e^{\frac{-\vp}{h}}\Lc^*(x,hD)e^{\frac{\vp}{h}}$ have the same principal term with possibly different perturbation coefficients and with $\vp$ replaced by $-\vp$.

\subsection{Construction of CGO solutions}\label{ccs}
Next we use the Carleman estimate \eqref{Car_0} to construct CGO solutions for the equation $\Lc(x,D)u=0$. First let us state an existence result whose proof  is standard; see \cite{DOS,KRU2} for instance. 
\begin{proposition}\label{Prop_Existence}
Let $\Lc$ and $\vp$ be as defined in Proposition \ref{Prop: Interior Carleman Estimate}. Then for any $v \in L^{2}(\Omega)$ and small enough $h>0$ one has $u \in H^{2m}(\Omega)$ such that
\begin{equation}\label{Exist_0}
	e^{-\frac{\vp}{h}}\Lc(x,D)e^{\frac{\vp}{h}} u = v \quad \mbox{in }\Omega, \quad \mbox{with}\quad \|u\|_{H^{2m}_{\mathrm{scl}}(\Omega)} \leq Ch^{m}\|v\|_{L^2(\Omega)}.
\end{equation}
\end{proposition}

Let us now use Proposition \ref{Prop_Existence} to construct a solution for $\Lc(x,D)u=0$ of the form
\begin{align}\label{CGO_0}
	u(x;h) &= e^{\frac{\vp + \I \psi}{h}}\left(a_0(x) + ha_1(x) + \dots + h^{m-1}a_{m-1}(x) + r(x;h)\right),\\
	\notag &=e^{\frac{\vp + \I \psi}{h}}\lb a(x;h) + r(x;h)\rb, \quad\mbox{where } a(x;h)= \sum_{j=0}^{m-1} h^ja_j(x).
\end{align}
Here $\{a_j(x)\}$ and $r(x;h)$ will be determined later.
We choose $\psi(x) \in C^{\infty}(\overline{\Omega})$ in such a way that
$p_{0,\vp}(x,\nabla\psi) =0$. This implies $
|\nabla \vp| = |\nabla \psi|$  and $\nabla\vp\cdot\nabla\psi = 0$ in $\Omega$.

We calculate the term $\Lc(x,D)e^{\frac{\vp+i\psi}{h}}a(x;h)$ in $\Omega$.
Due to the choices of $\vp$ and $\psi$ we have $\nabla(\vp+\I\psi) \cdot \nabla(\vp+\I\psi) = 0$ in $\Omega$ and thus we obtain

\begin{equation}\label{CGO_1}
\begin{aligned}
e^{-\frac{\vp+i\psi}{h}}\Lc(x,D) e^{\frac{\vp+i\psi}{h}}a(x;h)
=& \left(-\frac{1}{h}T - \Delta_x \right)^{m} a(x;h) \\
&- \sum_{i_1,\cdots,i_{m-2}=1}^n \overline{A}^{\:m-2}_{i_1\dots i_{m-2}}H^{m-2}_{i_1\dots i_{m-2}}\left(\frac{1}{h}T + \Delta_x\right)a(x;h)\\
&+ \sum_{j=1}^{m-1}\sum_{i_1,\cdots,i_j=1}^n A^j_{i_1\dots i_{j}}H^j_{i_1\dots i_j}a(x;h) + q(x)a(x;h),
\end{aligned}
\end{equation}
where
\begin{equation}\label{term_H}
\begin{gathered}
	T = 2\nabla_x(\vp+\I\psi)\cdot\nabla_x + \Delta_x(\vp+\I\psi),\\
	H^j_{i_1\dots i_j} 
	= e^{-\frac{\vp+\I\psi}{h}}D^j_{i_1\dots i_j}e^{\frac{\vp+\I\psi}{h}}
	=\prod_{k=1}^{j} \left(\frac{1}{h}D_{i_k}(\vp+\I\psi) + D_{i_k}\right). 
\end{gathered}
\end{equation}
Let us introduce one more notation $M_j$ denoting the operator with coefficient $\mathcal{O}(h^{-m+j})$ in
\[	\left(-\frac{1}{h}T-\Delta_x\right)^m - \sum_{i_1,\cdots,i_{m-2}=1}^n \overline{A}^{\:m-2}_{i_1\dots i_{m-2}}H^{m-2}_{i_1\dots i_{m-2}}\left(\frac{1}{h}T + \Delta_x\right)\\
+ \sum_{j=1}^{m-1}\sum_{i_1,\cdots,i_j=1}^n A^j_{i_1\dots i_{m-2}}H^j_{i_1\dots i_j}.
\]
We have 
\[\begin{aligned}
M_0 =& (-T)^m,\\
M_1 =& \sum_{l=0}^{m-1}(-1)^{m} T^{m-1-l}\circ \Delta_x\circ T^{l}
- \sum_{i_1,\cdots,i_{m-2}=1}^n \overline{A}^{\:m-2}_{i_1\dots i_{m-2}}\left(\prod_{k=1}^{m-2}D_{i_k}(\vp+\I\psi)\right)\circ T\\
&+ \sum_{i_1,\cdots,i_{m-1}=1}^n A^{m-1}_{i_1\dots i_{m-1}}\left(\prod_{k=1}^{m-1}D_{i_k}(\vp+\I\psi)\right),\\
&\vdots\\
M_{m-1} =& \sum_{k=0}^{m-1}(-1)^{m} (\Delta_x)^k\circ T \circ (\Delta_x)^{m-1-k}\\
&- \sum_{i_1,\cdots,i_{m-2}=1}^n \overline{A}^{\:m-2}_{i_1\dots i_{m-2}}\left[D^{m-2}_{i_1\dots i_{m-2}}\circ T + \sum_{l=0}^{m-3} D^l_{i_1\dots i_{l}} \left(D_{i_{l+1}}(\vp+\I\psi)D^{m-3-l}_{i_{l+2}\dots i_{m-3}}\right) \circ\Delta_x\right]\\
&+ \sum_{i_1,\cdots,i_{m-1}=1}^n A^{m-1}_{i_1\dots i_{m-1}}\left(\sum_{l=0}^{m-2} D^l_{i_1\dots i_{l}}\left(D_{i_{l+1}}(\vp+\I\psi)D^{m-2-l}_{i_{l+1}\dots i_{m-2}}\right)\right).
\end{aligned}
\]
Separating orders of $h$ in \eqref{CGO_1} we obtain the following set of equations for $a_j(x)$ as
\begin{eqnarray}
\mathcal{O}(h^{-m}):&\quad 
T^m a_0(x) =& 0, \quad \mbox{in }\Omega,	\label{CGO_2.0}\\
\mathcal{O}(h^{-m+j}):&\quad 
T^m a_j(x) =& -\sum_{k=1}^{j} M_{k} a_{j-k}(x), \quad \mbox{in }\Omega, \quad 1\leq j\leq m-1.	\label{CGO_2.j}
\end{eqnarray}
Existence of smooth solutions of \eqref{CGO_2.j} is well-known; see \cite{DOS}.
We give an explicit form for the smooth solution $a_0(x)$ solving \eqref{CGO_2.0} in the next section, which will play an important role in obtaining the generalized MRT of the coefficients.
\begin{remark}\label{Rem_isotropy_1}
    Note that in \eqref{CGO_1}, we crucially use the fact that $A^m=i_{\delta}\overline{A}^{\: m-2}$, where $\overline{A}^{\: m-2}$ is a symmetric $m-2$ tensor. For a general symmetric $m$ tensor $A^m$, we would get an inhomogeneous transport equation for the amplitude $a_0$ instead of the one in \eqref{CGO_2.0}. Getting a homogeneous equation \eqref{CGO_2.0} helps us to obtain a large enough class of solutions $a_0$ as in Lemma \ref{Prop4.1} and this plays a very important role in the proof of Theorem \ref{mainresult_1}.
\end{remark}


Choosing $a_0(x),\cdots,a_{m-1}(x) \in C^{\infty}(\Omega)$ satisfying \eqref{CGO_2.0}, \eqref{CGO_2.j}, we see that
\[  e^{-\frac{\vp+\I\psi}{h}}\Lc(x,D)e^{\frac{\vp+\I\psi}{h}}a(x;h)\simeq\Oc(1).
\]
Now if $u(x;h)$ as in \eqref{CGO_0} is a solution of $\Lc(x,D)u(x;h)=0$ in $\Omega$, we see that
\begin{align}\label{CGO_3}
0 = e^{-\frac{\vp+\I\psi}{h}}\Lc(x,D) u = e^{-\frac{\vp+\I\psi}{h}}\Lc(x;D) e^{\frac{\vp+\I\psi}{h}} \left(a(x;h) + r(x;h)\right).
\end{align}
This implies 
\begin{align}
e^{-\frac{\vp+\I\psi}{h}}\Lc(x,D) e^{\frac{\vp+\I\psi}{h}}r(x;h) = F(x;h), \quad \mbox{for some }F(x;h) \in L^2(\Omega), \quad \mbox{for all } h>0 \mbox{ small}.
\end{align}
By our choices, $a_j(x)$ annihilates all the terms of order $h^{-m+j}$ in $e^{-\frac{\vp+\I\psi}{h}}\Lc(x;D) e^{\frac{\vp+\I\psi}{h}}a(x;h)$ in $\Omega$ for $j=0,\dots,m-1$. Thus we get $\|F(x,h)\|_{L^2(\Omega)} \leq C$, where $C>0$ is uniform in $h$ for $h\ll 1$. 

Using Proposition \ref{Prop_Existence} we have the existence of $r(x;h) \in H^{2m}(\Omega)$ solving
\[e^{-\frac{\vp+\I\psi}{h}}\Lc(x,D)  e^{\frac{\vp+\I\psi}{h}}r(x;h) = F(x;h),
\]
with the estimate
\[	\|r(x;h)\|_{H^{2m}_{\mathrm{scl}}(\Omega)} \leq C h^m, \quad \mbox{for }h>0\mbox{ small enough}.
\]
Let $\widetilde{\Lc}(x,D)$ be the operator in \eqref{operator} with coefficients $(\widetilde{A}^j)_{1\leq j\leq m}$ and $\widetilde{q}$ and we denote the formal $L^2$ adjoint of $\wt{\Lc}$ by $\wt{\Lc}^{*}$. 
Summing up the above analysis, we state the following result on the existence of CGO solutions for the operators $\Lc(x,D)$ and $\wt{\Lc}^{*}(x,D)$ in $\Omega$.
\begin{proposition}\label{Prop_solvibility}
Let $h>0$ small enough and $\vp,\psi \in C^{\infty}(\overline{\Omega})$ satisfy $p_{0,\vp}(x,\nabla\psi) =0$. There are suitable choices of $a_0(x), \dots, a_{m-1}(x), b_0(x), \dots, b_{m-1}(x) \in C^{\infty}(\overline{\Omega})$ 
and $r(x;h), \wt{r}(x;h) \in H^{2m}(\Omega)$ such that
\begin{equation}\label{cgos}
\begin{aligned}
u(x;h) =& e^{\frac{\vp + \I \psi}{h}}\left(a_0(x) + ha_1(x) + \dots + h^{m-1}a_{m-1}(x) + r(x;h)\right),\\
v(x;h) =& e^{\frac{-\vp + \I \psi}{h}}\left(b_0(x) + hb_1(x) + \dots + h^{m-1}b_{m-1}(x) + \wt{r}(x;h)\right)
\end{aligned}
\end{equation}
solving $\Lc(x,D)u(x;h) = 0 = \wt{\Lc}^{*}(x,D)v(x;h)$ in $\Omega$ for $h>0$ small enough, with the estimates $\|r(x;h)\|_{H^{2m}_{\mathrm{scl}}(\Omega)}, \|\wt{r}(x;h)\|_{H^{2m}_{\mathrm{scl}}(\Omega)} \leq C h^{m}$.
Here, in particular, $a_0(x)$ and $b_0(x)$ solve the transport equations
\begin{equation}\label{Transport_0}
T^m a_0(x) = 0 \quad \mbox{and}\quad \overline{T}^m b_0(x) = 0, \quad \mbox{in }\Omega.
\end{equation}
\end{proposition}

\subsection{Solutions of transport equations} 
We begin with a few computations already done in \cite{DOS,DKSU}.
First let us assume that the origin $0 \in \Rb^n$ is outside $\overline{\Omega}$.
Let us write
\begin{equation}\label{cylindrical_coordinates}
	x=(x_1,\cdots,x_n)=(x_1,x') = (x_1,r,\theta) \in \R^n, \quad n\geq 3,
\end{equation}
where $(r,\theta)$ is the polar coordinate representation of $x^{\prime} \in \Rb^{n-1}$, centered at $0$, with $r=|x^{\prime}|$, $\theta \in \mathbb{S}^{n-2}$ such that $x'=r\theta$.
Note that, with a translation of coordinates we can choose any point $ (x_1,x'_0)$ to be the origin $0$ as long as $(x_1,x_0')\notin \overline{\Omega}$. Let us denote the complex variable $z=x_1+\I r$. 
Define the $\theta$-slice of $\Omega$ in the complex plane: 
\begin{equation}\label{Omega_theta}
    \Omega_{\theta} := \{(x_1+ir) \in \Cb \st (x_1,r,\theta) \in \Omega\}.
\end{equation}
We choose
\begin{equation}\label{cylindrical_coordinates_1}
	\vp = x_1 = \text{Re}(z), \qquad 
	\psi = r = \text{Im}(z).
\end{equation}
In this new system of coordinates, we have
\[\begin{aligned}
	&\vp+\I \psi=z; \qquad
	&&\n_{x}(\vp+\I \psi)=e_1+ie_r, \quad \mbox{ where } e_r=\frac{x'}{|x'|};\\
	& \Delta (\vp+\I \psi) = \frac{-2(n-2)}{(z-\bar{z})};\qquad
	&&T = 4\left(\PD_{\overline{z}} - \frac{(n-2)}{2(z-\bar{z})}\right).
\end{aligned}\]

The first transport equations \eqref{Transport_0} in $\Omega$ for $a_0(x)$ and $\wt{a_0}(x)$ reads in this new coordinate system as
\begin{eqnarray}
	T^m a_0 = 4^m\left(\partial_{\overline{z}}- \frac{(n-2)}{2(z-\bar{z})}\right)^m a_0 = 0,\label{transport_1}\\
	\overline{T^m} b_0 = \overline{4^m\left(\partial_{\overline{z}} - \frac{(n-2)}{2(z-\bar{z})}\right)^m} b_0 = 0.\label{transport_2}
\end{eqnarray}

Observe that if $a_0$ is a solution of \eqref{transport_1} then so is $a_0(x_1,r)g(\theta)$ for any smooth function $g(\theta)$. Similarly for \eqref{transport_2}, $b_0(z)g(\theta)$ is a solution if $b_0$ solves \eqref{transport_2} and $g$ is any smooth function.
Therefore, it is reasonable to study the general solutions $a_0$ and $b_0$ as functions of the $z$ variable. 
\begin{lemma}\label{Prop4.1}
	The general solution of \eqref{transport_1} 
	is $a_0(z) = \sum_{k=1}^m(z-\bar{z})^{(2k-n)/2} h_k(z)$, where for each $k$ with $1\le k\le m$,  $h_k(z)$ is an arbitrary holomorphic function.
\end{lemma}
\begin{proof}
	We prove the result by induction. For $m=1$, \eqref{transport_1} is 
\[ \left(\partial_{\bar{z}}- \frac{(n-2)}{2(z-\bar{z})}\right) a_0(z) = 0.\]
A direct calculation shows that 
\[
0=\left(\partial_{\bar{z}}- \frac{(n-2)}{2(z-\bar{z})}\right) a_0(z)= \PD_{\bar{z}} \left( (z-\bar{z})^{(n-2)/2} a_0  \right).
\] 
This implies $(z-\bar{z})^{(n-2)/2} a_0 $ is a holomorphic function.
Therefore $a_0(z) = (z-\bar{z})^{(2-n)/2} h(z)$ for a homolomorphic function $h$. This completes the proof for $m=1$.
Assume that the statement of the lemma is true for some $m\geq 1$. Now we show that any solution of $ T^{m+1} a_0=0 $ is of the form $ a_0(z)= \sum\limits_{k=1}^{m+1}(z-\bar{z})^{(2k-n)/2} h_{k}(z)$ for $ m+1 $ holomorphic functions $ h_k(z),\, 1\le k\le m+1 $.

Let $Ta_0(z)=b_0(z)$ and we have  $T^m b_0(z)=0 $, since  $T^{m+1} a_0(z)=0 $. By induction hypothesis we have that 
\[ b_0(z) = \sum\limits_{k=1}^{m}(z-\bar{z})^{(2k-n)/2} h_{k}(z). \]
Multiplying  the equation $  Ta_0(z)=b_0(z) $ by $ (z-\bar{z})^{(n-2)/2} $ we obtain
\begin{align*}
&(z-\bar{z})^{(n-2)/2}Ta_0(z)= (z-\bar{z})^{(n-2)/2}b_0(z)\\
&4\PD_{\bar{z}} \left( (z-\bar{z})^{(n-2)/2} a_0(z) \right)= \sum\limits_{k=1}^{m}(z-\bar{z})^{k-1} h_{k}(z)\\
&4\PD_{\bar{z}} \left( (z-\bar{z})^{(n-2)/2} a_0(z) \right)=- \PD_{\bar{z}}\left( \sum\limits_{k=1}^{m} \frac{1}{k}\,(z-\bar{z})^{k} h_{k}(z)\right)\\
&\PD_{\bar{z}} \left( (z-\bar{z})^{(n-2)/2} a_0(z) + \sum\limits_{k=1}^{m} \frac{1}{4k}\,(z-\bar{z})^{k} h_{k}(z)\right)=0.
\end{align*}
This implies $ (z-\bar{z})^{(n-2)/2} a_0(z) + \sum\limits_{k=1}^{m} \frac{1}{4k}\,(z-\bar{z})^{k} h_{k}(z)$ is holomorphic. Denoting this as $h_{0}(z)$, we get,  
\[a_0(z) =  \sum\limits_{k=1}^{m+1}\,(z-\bar{z})^{(2k-n)/2} \tilde{h}_{k}(z),  \]
here $ \tilde{h}_1(z)= h_0(z) $ and $ \tilde{h}_i(z)= -\frac{1}{4(i-1)}h_{i-1}(z) $ for $ 2\le i\le m+1 $.
By induction, we have shown that $\sum\limits_{k=1}^{m}\,(z-\bar{z})^{(2k-n)/2} \tilde{h}_{k}(z)$, where $\tilde{h}_{k}$'s  are holomorphic, is the general solution of \eqref{transport_1}. This finishes the proof.
\end{proof}

\subsection{Boundary Carleman estimates}
Recall that in $\Nc_{\Gamma_F}$, the Neumann data is not given over the entire $\partial\Omega$. 
In order to estimate the solution on the part of the boundary where we do not have information, we use a boundary Carleman estimate. Unlike the interior Carleman estimate which follows in a straightforward manner from prior work, the boundary Carleman estimate is slightly more involved, since we are dealing with anisotropic coefficients. Specifically, we need to use semiclassical elliptic regularity estimates to absorb lower order terms. 
Let us consider $\Lc(x,D)$ as in \eqref{operator}. 
We prove the following boundary Carleman estimate.
\begin{proposition}\label{bdy_Car_Est}
Let $\Lc(x,D)$ be as in  \eqref{operator} and $\vp$ be a limiting Carleman weight for the semiclassical Laplacian on $\widetilde{\O}$. Then for $u \in \mathcal{D}(\mathcal{L}(x,D))$ (see \eqref{domain}) and $0<h\ll 1$, we have
\begin{equation}\label{bdy_carlemanestimate}
\begin{aligned}
\sum_{k=0}^{m-1} h^{\frac{3}{2}+k}\left\lVert |\partial_{\nu}\vp|^{\frac{1}{2}} e^{-\frac{\vp}{h}}\partial_{\nu}(-h^2\Delta)^{m-k-1}u \right\rVert_{L^2(\partial_{+}\Omega)}
+ \sum_{k=0}^{[\frac{m}{2}]} h^{m-k} \left\lVert e^{-\frac{\vp}{h}}(-h^2\Delta)^ku \right\rVert_{H^1_{\mathrm{scl}}(\Omega)}
\\
\leq
Ch^{2m}\left\lVert e^{-\frac{\vp}{h}}\mathcal{L}(x,D)u\right\rVert_{L^2(\Omega)}
+ C\sum_{k=0}^{m-1} h^{\frac{3}{2}+k}\left\lVert |\partial_{\nu}\vp|^{\frac{1}{2}} e^{-\frac{\vp}{h}}\partial_{\nu}(-h^2\Delta)^{m-k-1}u \right\rVert_{L^2(\partial_{-}\Omega)},
\end{aligned}
\end{equation}
	where the constant $C$ is independent of  both $u$ and $h$. 
\end{proposition}
\begin{proof}
We start by recalling the boundary Carleman estimate for the conjugated semiclassical operator $e^{-\vp/h}(-h^2\Delta)^me^{\vp/h}$ in \cite[Equation (2.23)]{Ghosh-Krishnan}:
\begin{equation}\label{BCE_0}
\begin{aligned}
\sum_{k=0}^{m-1} h^{\frac{3}{2}+k}\left\lVert |\partial_{\nu}\vp|^{\frac{1}{2}} e^{-\frac{\vp}{h}}\partial_{\nu}(-h^2\Delta)^{m-k-1}u \right\rVert_{L^2(\partial_{+}\Omega)}
+ \sum_{k=0}^{[\frac{m}{2}]} h^{m-k} \left\lVert e^{-\frac{\vp}{h}}(-h^2\Delta)^ku \right\rVert_{H^1_{\mathrm{scl}}(\Omega)}
\\
\leq
C\left\lVert e^{-\frac{\vp}{h}}(-h^2\Delta)^mu\right\rVert_{L^2(\Omega)}
+ C\sum_{k=0}^{m-1} h^{\frac{3}{2}+k}\left\lVert |\partial_{\nu}\vp|^{\frac{1}{2}} e^{-\frac{\vp}{h}}\partial_{\nu}(-h^2\Delta)^{m-k-1}u \right\rVert_{L^2(\partial_{-}\Omega)},
\end{aligned}
\end{equation}

for all $u \in \Dc(\Lc(x,D))$. 
The lower order terms in the operator $\Lc(x,D)$ consists of anisotropic terms, whereas \eqref{BCE_0} consists of terms only involving powers of the Laplacian. Therefore, as mentioned before, we need semi-classical versions of the standard elliptic regularity estimates. We could not find a ready reference for such estimates and therefore we prove this in Appendix \ref{Appendix_elliptic_regu}.

For $u \in H^{2m}(\Omega)$ and for any $k\leq 2m$, we have
\begin{equation}\label{claim1}
	\left\| e^{-\frac{\vp}{h}}h^k D^k_{i_1\dots i_k} u\right\|_{L^2(\Omega)}
	\leq C\left\|e^{-\frac{\vp}{h}}u \right\|_{H^k_{\mathrm{scl}}(\Omega)}.
\end{equation}
This follows from the identity: For $u\in H^{1}(\O)$,
\[	e^{-\frac{\vp}{h}}hD_ju 
=	\left[e^{-\frac{\vp}{h}} hD_j e^{\frac{\vp}{h}}\right] \left(e^{-\frac{\vp}{h}}u\right)
=	\left[\left(D_j\vp\right) + hD_j\right] \left(e^{-\frac{\vp}{h}}u\right). 
\]
Iterating this $k$ times, we obtain for $u\in H^{k}(\O)$:
\[	e^{-\frac{\vp}{h}}h^k D^k_{i_1\dots i_k} u
= \left(\prod_{j=1}^{k} \left[\left(D_{i_j}\vp\right) + hD_{i_j}\right]\right)\left(e^{-\frac{\vp}{h}}u\right),\mbox{ for } u \in H^k(\Omega).
\]
Thus for all $u \in H^{2m}(\Omega)$ and for $k\leq 2m$, we have
\[	\left\| e^{-\frac{\vp}{h}}h^k D^k_{i_1\dots i_k} u \right\|_{L^2(\Omega)} 
\leq \left\| \left(\prod_{j=1}^{k} \left[\left(D_{i_j}\vp\right) + hD_{i_j}\right]\right) \left(e^{-\frac{\vp}{h}} u\right) \right\|_{L^2(\Omega)}
\leq C \left\| e^{-\frac{\vp}{h}}u \right\|_{H^k_{\mathrm{scl}}(\Omega)}.
\]

Next we show that for any $u \in H^{2m}(\Omega)$ with $(-\Delta)^ju|_{\partial\Omega}  = 0$ for $j=0,1,\dots,m-1$,
\begin{equation}\label{claim}
\left\| e^{-\frac{\vp}{h}}u\right\|_{H^k_{\mathrm{scl}}(\Omega)}
\leq C\sum_{j=0}^{[\frac{k}{2}]}\left\|e^{-\frac{\vp}{h}}(-h^2\Delta)^j u \right\|_{H^1_{\mathrm{scl}}(\Omega)},\quad \mbox{for }k=0,1,\dots,m-1,
\end{equation}
 for some constant $C>1$. Here $[\cdot]$ is the greatest integer function. 

We prove this inequality by induction. 
We repeatedly use the following semi-classical elliptic regularity result (proved in Appendix \ref{Appendix_elliptic_regu}): 
\begin{equation}\label{Ellip_regu}
\| u \|_{H^{j+2}_{\mathrm{scl}}(\Omega)} 
\leq C\left(\|(-h^2\Delta)u\|_{H^{j}_{\mathrm{scl}}(\Omega)} + \|u\|_{L^2(\Omega)}\right),\mbox{ for } u \in H^{j+2}(\Omega) \cap H^1_0(\O) \mbox{ and  } 0\leq j\leq 2m-2.
\end{equation}
Going back to the proof of \eqref{claim}, observe that for $k=0,1$, this estimate trivially holds.

Let us assume that the claim is true for $j=0,1,\dots, k$ for some $k < m-1$. That is, for any $0\leq j \leq k$ we have
\[  \left\| e^{-\frac{\vp}{h}}u\right\|_{H^j_{\mathrm{scl}}(\Omega)}
\leq C\sum_{l=0}^{[\frac{j}{2}]}\left\|e^{-\frac{\vp}{h}}(-h^2\Delta)^l u \right\|_{H^1_{\mathrm{scl}}(\Omega)},
\]
for $u \in H^{2m}(\Omega)$ with $\lb -\Delta\rb^{l}  u|_{\PD \O} = 0$ for $l=0,1,\dots,[j/2]$.
For $k+1$, we have, using \eqref{Ellip_regu},
\[\begin{aligned}
\left\| e^{-\frac{\vp}{h}} u\right\|_{H^{k+1}_{\mathrm{scl}}(\Omega)}
\leq& C\left\| (-h^2\Delta) e^{-\frac{\vp}{h}} u\right\|_{H^{k-1}_{\mathrm{scl}}(\Omega)} + C\left\|e^{-\frac{\vp}{h}}u\right\|_{L^2(\Omega)} \mbox{ since }e^{-\frac{\vp}{h}} u|_{\partial\Omega} = 0,\\
\leq& C\lb \left\| e^{-\frac{\vp}{h}} \left[(-h^2\Delta) + 2\nabla\vp\cdot h\nabla + \left(h\Delta\vp - |\nabla \vp|^2\right)\right] u\right\|_{H^{k-1}_{\mathrm{scl}}(\Omega)}+ \left\|e^{-\frac{\vp}{h}}u\right\|_{L^2(\Omega)}\rb \\
\leq& C\left(\left\| e^{-\frac{\vp}{h}} (-h^2\Delta) u\right\|_{H^{k-1}_{\mathrm{scl}}(\Omega)} 
+ \left\| e^{-\frac{\vp}{h}} (hD) u\right\|_{H^{k-1}_{\mathrm{scl}}(\Omega)}
+ \left\| e^{-\frac{\vp}{h}} u\right\|_{H^{k-1}_{\mathrm{scl}}(\Omega)} \right)\\
\leq& C\left(\left\| e^{-\frac{\vp}{h}} (-h^2\Delta) u\right\|_{H^{k-1}_{\mathrm{scl}}(\Omega)} 
+ \left\| \left(D\vp + hD\right)\left(e^{-\frac{\vp}{h}} u\right)\right\|_{H^{k-1}_{\mathrm{scl}}(\Omega)}
+ \left\| e^{-\frac{\vp}{h}} u\right\|_{H^{k-1}_{\mathrm{scl}}(\Omega)} \right)\\
\leq& C\left(\left\| e^{-\frac{\vp}{h}} (-h^2\Delta) u\right\|_{H^{k-1}_{\mathrm{scl}}(\Omega)} 
+ \left\| e^{-\frac{\vp}{h}} u\right\|_{H^{k-1}_{\mathrm{scl}}(\Omega)} 
+ \left\| e^{-\frac{\vp}{h}} u\right\|_{H^{k}_{\mathrm{scl}}(\Omega)}\right). 
\end{aligned}
\]
Using the induction assumption along with the fact that $e^{-\frac{\vp}{h}}(-h^2\Delta)^lu |_{\partial\Omega} = 0$, for $l=0,1,\dots, [\frac{k+1}{2}]$ we get
\[\begin{aligned}
\left\| e^{-\frac{\vp}{h}} u\right\|_{H^{k+1}_{\mathrm{scl}}(\Omega)}
\leq& C\left( \sum_{j=0}^{[\frac{k-1}{2}]}\left\| e^{-\frac{\vp}{h}} (-h^2\Delta)^j(-h^2\Delta)u\right\|_{H^1_{\mathrm{scl}}(\Omega)}
+ \sum_{j=0}^{[\frac{k-1}{2}]}\left\| e^{-\frac{\vp}{h}} (-h^2\Delta)^ju\right\|_{H^1_{\mathrm{scl}}(\Omega)}\right.\\
&\qquad\left. + \sum_{j=0}^{[\frac{k+1}{2}]}\left\| e^{-\frac{\vp}{h}} (-h^2\Delta)^ju\right\|_{H^1_{\mathrm{scl}}(\Omega)}\right)\\
\leq& C\sum_{j=0}^{[\frac{k+1}{2}]}\left\| e^{-\frac{\vp}{h}} (-h^2\Delta)^ju\right\|_{H^1_{\mathrm{scl}}(\Omega)},
\end{aligned}
\]
for $u \in H^{2m}(\Omega)$ with $(-h^2\Delta)^j u|_{\partial\Omega} = 0$ for $0\leq j \leq [\frac{k+1}{2}]$.
The proof of \eqref{claim} is now complete.

Combining \eqref{claim1} and \eqref{claim}, we obtain, for $u \in H^{2m}(\Omega)$ with $(-\Delta)^ju|_{\partial\Omega}  = 0$ for $j=0,1,\dots,m-1$,
\Beq\label{Claim3}
	\left\| e^{-\frac{\vp}{h}}h^k D^k_{i_1\dots i_k} u\right\|_{L^2(\Omega)}
	\leq C\sum_{j=0}^{[\frac{k}{2}]}\left\|e^{-\frac{\vp}{h}}(-h^2\Delta)^j u \right\|_{H^1_{\mathrm{scl}}(\Omega)},\quad \mbox{for }k=0,1,\dots,m-1.
\Eeq

Using \eqref{Claim3}, for $h>0$ small enough we have,
\begin{equation}\label{BCE_1}
\begin{aligned}
\left\| e^{-\frac{\vp}{h}} h^{2m}\overline{A}^{\:m-2}_{i_1 \dots i_{m-2}} D^{m-2}_{i_1\dots i_{m-2}}\circ\Delta u\right\|_{L^2(\Omega)}&\leq Ch^m\left\| e^{-\frac{\vp}{h}} h^{m-2}D^{m-2}_{i_1\dots i_{m-2}}(-h^2\Delta u)\right\|_{L^2(\Omega)}\\
&\leq Ch^m\sum_{j=0}^{[\frac{m-2}{2}]}\left\| e^{-\frac{\vp}{h}} (-h^2\Delta)^j(-h^2\Delta u)\right\|_{H^1_{\mathrm{scl}}(\Omega)}\\
&\leq Ch^m\sum_{j=1}^{[\frac{m}{2}]}\left\| e^{-\frac{\vp}{h}} (-h^2\Delta)^ju\right\|_{H^1_{\mathrm{scl}}(\Omega)}\\
&\leq Ch\sum_{j=0}^{[\frac{m}{2}]}h^{m-j}\left\| e^{-\frac{\vp}{h}} (-h^2\Delta)^ju\right\|_{H^1_{\mathrm{scl}}(\Omega)}.
\end{aligned}
\end{equation}
Similarly the other perturbations for $k=1,\dots,m-1$ can be estimated as
\begin{equation}\label{BCE_2}
\begin{aligned}
&\left\| e^{-\frac{\vp}{h}} h^{2m}A^{k}_{i_1 \dots i_{k}}D^{k}_{i_1\dots i_{k}}u\right\|_{L^2(\Omega)}
\leq Ch^{2m-k}\left\| e^{-\frac{\vp}{h}} h^{k}D^{k}_{i_1\dots i_{k}}u\right\|_{L^2(\Omega)}\\
&\leq Ch^{2m-k}\sum_{j=1}^{[\frac{k}{2}]}\left\| e^{-\frac{\vp}{h}} (-h^2\Delta)^ju\right\|_{H^1_{\mathrm{scl}}(\Omega)}
\leq Ch^{m-k+1}\sum_{j=0}^{[\frac{m}{2}]}h^{m-j}\left\| e^{-\frac{\vp}{h}} (-h^2\Delta)^ju\right\|_{H^1_{\mathrm{scl}}(\Omega)}.
\end{aligned}
\end{equation}
Therefore by the triangle inequality we obtain
\begin{equation}\label{BCE_3}
\begin{aligned}
\left\|e^{-\frac{\vp}{h}} h^{2m}\Lc(x,D)u\right\|_{L^2(\Omega)}
\geq& \left\|e^{-\frac{\vp}{h}} (-h^2\Delta)^m u\right\|_{L^2(\Omega)}
- \left\| e^{-\frac{\vp}{h}} h^{2m}\overline{A}^{\:m-2}_{i_1 \dots i_{m-1}}D^{m-2}_{i_1\dots i_{m-2}}\circ\Delta u\right\|_{L^2(\Omega)}\\
&- \left\| e^{-\frac{\vp}{h}} h^{2m}A^{k}_{i_1 \dots i_{k}}D^{k}_{i_1\dots i_{k}}u\right\|_{L^2(\Omega)}
- \left\| e^{-\frac{\vp}{h}} h^{2m}qu\right\|_{L^2(\Omega)}\\
\geq& \left\|e^{-\frac{\vp}{h}} (-h^2\Delta)^m u\right\|_{L^2(\Omega)}\\
&- C\sum_{k=0}^{m} h^{m-k+1}\sum_{j=0}^{[\frac{m}{2}]}h^{m-j}\left\| e^{-\frac{\vp}{h}} (-h^2\Delta)^ju\right\|_{H^1_{\mathrm{scl}}(\Omega)}\\
\geq& \left\|e^{-\frac{\vp}{h}} (-h^2\Delta)^m u\right\|_{L^2(\Omega)}
-Ch\sum_{j=0}^{[\frac{m}{2}]}h^{m-j}\left\| e^{-\frac{\vp}{h}} (-h^2\Delta)^ju\right\|_{H^1_{\mathrm{scl}}(\Omega)}.
\end{aligned}
\end{equation} 
Now combining \eqref{BCE_3} with \eqref{BCE_0},  we get the desired estimate.
\end{proof}

\section{Determination of the coefficients} \label{determination}
In this section we prove Theorem \ref{mainresult_1} using the CGO solutions established in Section \ref{Sec_Prelim} along with the boundary Carleman estimate in Proposition \ref{bdy_Car_Est}. We start by deriving certain integral identities for the difference of the operators $\Lc(x,D)$ and $\wt{\Lc}(x,D)$.
\subsection{Integral identities}
	We recall that
	\[\Lc(x,D) = 
		(-\Delta)^m + \sum\limits_{i_1,\cdots, i_{m-2}=1}^{n} \overline{A}^{\:m-2}_{i_{1}\cdots i_{m-2}}(x) D^{m-2}_{i_{1}\cdots i_{m-2}}\circ \lb -\Delta\rb+ \sum\limits_{j=1}^{m-1} \sum\limits_{i_1,\cdots, i_{j}=1}^{n}A_{i_{1}\cdots i_{m-j}}^{m-j}(x) D^{m-j}_{i_{1}\cdots i_{m-j}}+ q(x),
	\]
where $\overline{A}^{\:m-2}$ is a smooth symmetric $m-2$ tensor field, $A^{m-j}$ are smooth symmetric $(m-j)$ tensor fields in $\Omega$ for $1\leq j\leq m-1$ and $q \in L^{\infty}(\Omega)$. Applying integration by parts, we have the following integral identity 
\begin{equation}\label{integralidentity}
\begin{aligned}
\int_{\Omega} &\left(\Lc(x,D)u\right)\overline{v}\, \D x - \int_{\Omega} u\, \overline{\Lc^{*}(x,D)v}\,\D x\\
= &\sum_{k=0}^{m-1}\int_{\partial\Omega} \left((-\Delta)^{(m-k-1)}u\right) \overline{\left(\partial_{\nu} (-\Delta)^{k}v\right)}\,\D S
-\sum_{k=0}^{m-1}\int_{\partial\Omega} \left(\partial_{\nu} (-\Delta)^{(m-k-1)}u\right) \overline{\left((-\Delta)^{k}v\right)}\,\D S\\
&-\I \sum_{k=1}^{m-2}\sum_{i_1,\cdots,i_{m-2}=1}^{n} \int_{\partial\Omega} (-1)^k \nu_{i_{m-1-k}}\left(D^{m-2-k}_{i_1 \dots i_{m-2-k}} \circ \Delta u\right) D^{k-1}_{i_{m-k} \dots i_{m-2}}\left(\overline{A}^{\:m-2}_{i_1 \dots i_{m-2}} \overline{v}\right) \,\D S\\
&+ \sum_{i_1,\cdots,i_{m-2}=1}^{n} \int_{\partial\Omega} (-1)^{m-1} \left(\partial_{\nu}u\right)\, D^{m-2}_{i_{1} \dots i_{m-2}}\left(\overline{A}^{\:m-2}_{i_1 \dots i_{m-2}} \overline{v}\right) \,\D S\\
&+ \sum_{i_1,\cdots,i_{m-2}=1}^{n} \int_{\partial\Omega} (-1)^{m} u \, \partial_{\nu}\left[D^{m-2}_{i_{1} \dots i_{m-2}}\left(\overline{A}^{\:m-2}_{i_1 \dots i_{m-2}} \overline{v}\right)\right] \,\D S\\
&+\I \sum_{j=1}^{m-1}\sum_{k=1}^{j}\sum_{i_1,\cdots,i_{j}=1}^{n} \int_{\partial\Omega} (-1)^k \nu_{i_{j-k+1}}\left(D^{j-k}_{i_1 \dots i_{j-k}}u\right) D^{k-1}_{i_{j-k+2} \dots i_{j}}\left(A^{j}_{i_1 \dots i_{j}} \overline{v}\right) \,\D S,
\end{aligned}
\end{equation}
where $u,v\in H^{2m}(\Omega)$ and $\D S$ is the surface measure on $\partial\Omega$.
Let $u,\widetilde{u}\in H^{2m}(\Omega)$ solve
\begin{equation}\label{eq_1}
	\begin{aligned}
		\Lc(x,D)u &=0\quad\mbox{in }\Omega\qquad \quad
		\mbox{and}\quad
	\widetilde{\Lc}(x,D)\widetilde{u} =0 \quad\mbox{in }\Omega,\\
			\mbox{with }\quad (-\Delta)^l u|_{\partial\Omega} &= (-\Delta)^l \widetilde{u}|_{\partial\Omega},\quad \mbox{for }l=0,1,\ldots,(m-1),
		\end{aligned}
	\end{equation}
where $\widetilde{\Lc}(x,D)$ is defined in Theorem \ref{mainresult_1} having perturbation coefficients to be $\widetilde{A}^j$ and $\widetilde{q}$ for $j=1,\cdots,m$.
From the assumption in Theorem \ref{mainresult_1} on the boundary data we get
	\begin{equation}\label{eq_2}
		\partial_\nu(-\Delta )^{m-l}u|_{\Gamma_F} = \partial_\nu(-\Delta )^{m-l}\widetilde{u}|_{\Gamma_F}
		\quad\mbox{for} \quad l=1,\ldots,m.
	\end{equation}
	Therefore from \eqref{integralidentity}, \eqref{eq_1} and \eqref{eq_2} we get
\begin{equation}\label{integralidentity_1}
\begin{aligned}
\int_{\Omega} \left(\Lc(x,D)(u-\widetilde{u})\right)\overline{v} \D x - \int_{\Omega} (u-\widetilde{u})\overline{\Lc^{*}(x,D)v}\D x
=\Bc, \quad \mbox{for any }v\mbox{ in }H^{2m}(\Omega),
\end{aligned}
\end{equation}
where $\Bc$ is given by 
\begin{equation}
    \begin{aligned}\label{boundary_term}
   \Bc=& 
-\sum_{k=0}^{m-1}\int_{\partial\Omega\setminus \Gamma_{F}} \left(\partial_{\nu} (-\Delta)^{m-k-1}(u-\wt{u})\right) \overline{\left((-\Delta)^{k}v\right)}\,\D S\\
&-\I \sum_{k=1}^{m-2}\sum_{i_1,\cdots,i_{m-2}=1}^{n} \int_{\partial\Omega} (-1)^k \nu_{i_{m-1-k}}\left(D^{m-2-k}_{i_1 \dots i_{m-2-k}} \circ \Delta (u-\wt{u})\right) D^{k-1}_{i_{m-k} \dots i_{m-2}}\left(\overline{A}^{\:m-2}_{i_1 \dots i_{m-2}} \overline{v}\right) \,\D S\\
&+ \sum_{i_1,\cdots,i_{m-2}=1}^{n} \int_{\partial\Omega} (-1)^{m-1} \left(\partial_{\nu}(u-\wt{u})\right)\, D^{m-2}_{i_{1} \dots i_{m-2}}\left(\overline{A}^{\:m-2}_{i_1 \dots i_{m-2}} \overline{v}\right) \,\D S\\
&+\I \sum_{j=1}^{m-1}\sum_{k=1}^{j}\sum_{i_1,\cdots,i_{j}=1}^{n} \int_{\partial\Omega} (-1)^k \nu_{i_{j-k+1}}\left(D^{j-k}_{i_1 \dots i_{j-k}}(u-\wt{u})\right) D^{k-1}_{i_{j-k+2} \dots i_{j}}\left(A^{j}_{i_1 \dots i_{j}} \overline{v}\right) \,\D S.
\end{aligned}
\end{equation}

Let us choose $v\in H^{2m}(\Omega)$ as in \eqref{cgos} satisfying $\Lc^{*}(x,D)v = 0 $ in $\Omega$, and from \eqref{integralidentity_1} we get
\[  \Bc 
= \int_{\Omega} \left(\Lc(x,D)(u-\widetilde{u})\right)\overline{v} \D x 
= \int_{\Omega} \overline{v} \left(\widetilde{\Lc}(x,D)- \Lc(x,D)\right)\widetilde{u} \, \D x,
\]
where the last equality follows from
\[	\Lc(x,D)(u-\widetilde{u}) = \Lc(x,D)u - \Lc(x,D)\widetilde{u} = \widetilde{\Lc}(x,D)\widetilde{u} - \Lc(x,D)\widetilde{u} = \left(\widetilde{\Lc}(x,D)- \Lc(x,D)\right)\widetilde{u}.
\]
Now note that
\begin{equation}\label{differenceEven}
\begin{aligned}
\left(\widetilde{\Lc}(x,D)- \Lc(x,D)\right)\widetilde{u}
=&\sum_{i_1\dots i_{m-2}=1}^{n}(\widetilde{\overline{A}}^{\:m-2}_{i_1\dots i_{m-2}}-\overline{A}^{\:m-2}_{i_1\dots i_{m-2}})D^{m-2}_{i_1\dots i_{m-2}}\circ(-\Delta) \widetilde{u}\\
&+\sum_{k=1}^{m-1}\sum_{i_1\dots i_k=1}^{n}(\widetilde{A}^k_{i_1\dots i_k}-{A^k_{i_1\dots i_k}})D^k_{i_1\dots i_k}\widetilde{u}+(\widetilde{q}-q)\widetilde{u}.
\end{aligned}
\end{equation}
We recall our choices of $\widetilde{u},v \in H^{2m}(\Omega)$ using Proposition \ref{Prop_solvibility}:  
\begin{equation*}
\begin{aligned}
\widetilde{u}(x;h) =& e^{\frac{\vp + \I \psi}{h}}\left(a_0(x) + ha_1(x) + \dots + h^{m-1}a_{m-1}(x) + r(x;h)\right)\\
v(x;h) =& e^{\frac{-\vp + \I \psi}{h}}\left(b_0(x) + hb_1(x) + \dots + h^{m-1}b_{m-1}(x) + \wt{r}(x;h)\right),
\end{aligned}
\end{equation*}
solving $\widetilde{\Lc}(x,D)\widetilde{u}(x;h) = 0 = {\Lc}^*(x,D)v(x;h)$ in $\Omega$ for $0<h\ll 1$, with the bounds  $$\|r(x;h)\|_{H^{2m}_{\mathrm{scl}}(\Omega)},\, \|\wt{r}(x;h)\|_{H^{2m}_{\mathrm{scl}}(\Omega)} \leq ch^{m}.$$
Here $a_k, b_k$ are solutions of the transport equations  \eqref{CGO_2.0} and \eqref{CGO_2.j} for corresponding coefficients. Let us denote
\[ \begin{aligned}
\aK(x;h) =& a_0(x) + ha_1(x) + \dots + h^{m-1}a_{m-1}(x) + r(x;h)\\
\bK(x;h) =& b_0(x) + hb_1(x) + \dots + h^{m-1}b_{m-1}(x) + \wt{r}(x;h).
\end{aligned} \]
Using the CGO solutions $\widetilde{u}$ and $v$ we see that \eqref{differenceEven} implies
\begin{equation}\label{Int_id}
\begin{aligned}
\Bc =&\int_{\Omega} \left(\widetilde{\Lc}(x,D)- \Lc(x,D)\right)\widetilde{u}(x)\, \overline{v(x)}\, \D x \\
=&\sum_{i_1\dots i_{m-2}=1}^{n}\int_{\Omega}\left(\widetilde{\overline{A}}^{\:m-2}-\overline{A}^{\:m-2}\right)_{i_1\dots i_{m-2}} \left[H^{m-2}_{i_1\dots i_{m-2}}\left(-\frac{1}{h}T - \Delta\right) \aK(x;h)\right] \overline{\bK(x;h)}\,\D x \\
&+\sum_{k=1}^{m-1}\sum_{i_1\dots i_k=1}^{n} \int_{\Omega} (\widetilde{A^k}-A^k)_{i_1\dots i_k}\left[H^{k}_{i_1\dots i_k}\aK(x;h)\right] \, \overline{\bK(x;h)}\, \D x\\
&+\int_{\Omega}(\widetilde{q}-q)\aK(x;h) \, \overline{\bK(x;h)}\,\D x,
\end{aligned}
\end{equation}
where $T$ and $H^j$ are defined in \eqref{term_H}. 

In the next subsection we see that multiplying adequate powers of the small parameter $h$ to the boundary term $\Bc$ we can make the product small. This will in turn imply that multiplying sufficient powers of $h$ to the right hand side of \eqref{Int_id} we get integral identities relating the differences of the unknown coefficients i.e. $\left(\widetilde{A}^k-A^k\right)$ and $\left(\widetilde{q}-q\right)$ over $\Omega$. In Section \ref{Section_coeff_deter} we deal with the right hand side of \eqref{Int_id} in order to prove Theorem \ref{mainresult_1}. 

\subsection{Boundary estimates} In this section we will estimate the CGO solutions on the unknown part of the boundary $\partial\Omega\setminus\Gamma_F$ by the data from the known part of the boundary $\Gamma_F$.
Using these CGO forms of $\widetilde{u}, v$ and  the boundary Carleman estimate \eqref{Car_0} we prove the following results.
\begin{lemma}\label{Prop_boundary_est_1}
Let $u, \widetilde{u} \in H^{2m}(\Omega)$ solve \eqref{eq_1} with $\widetilde{u}$ as in the first equation of \eqref{cgos} and $v \in H^{2m}(\Omega)$ be as in the second equation of \eqref{cgos}. For $0<h\ll 1$ and $\Gamma_F \subset \partial\Omega$ be a neighborhood of $\partial_{-}\Omega$, we have the following estimate
\begin{equation}\label{boundary_est_1.1}
\lim\limits_{h\to 0} \int_{\partial\Omega\setminus\Gamma_F} h^{m-1}\left[\partial_{\nu} (-\Delta)^{{m-k-1}}(u-\widetilde{u})\right] \Delta^{k}\overline{v}\,\D S = 0 \quad \mbox{for} \quad 0\le k \le m-1.
\end{equation}
\end{lemma}
	\begin{proof}
	Observe that $\Gamma_F \subset \partial\Omega$ is an open neighborhood of $\partial_{-}\Omega = \{\PD_{\nu}\vp(x) \leq 0\}$. Using compactness of $\partial\Omega$, we get $\delta_0>0$ small enough, such that
	\begin{equation*}
		\partial\Omega \setminus \Gamma_F \ssubset \{ x \in \partial\Omega : \PD_{\nu}\vp(x) >\delta_0 > 0 \} \ssubset \partial_{+}\Omega.
	\end{equation*}
	From Proposition \ref{bdy_Car_Est} for each $0\leq k\leq m-1$, we get for all $w \in \mathcal{D}(\mathcal{L}(x,D))$,
	\begin{equation*}
		\begin{aligned}
			h^{2m}\left\lVert e^{-\frac{\varphi}{h}}\mathcal{L}(x,D)w\right\rVert_{L^2(\Omega)}
			+&\sum_{k=0}^{m-1} h^{\frac{3}{2}+k}\left\lVert |\partial_{\nu}\varphi|^{\frac{1}{2}} e^{-\frac{\varphi}{h}}\partial_{\nu}(-h^2\Delta)^{m-k-1}w \right\rVert_{L^2(\partial_{-}\Omega)}\\
			&\geq h^{\frac{3}{2}+k}\left\lVert |\partial_{\nu}\varphi|^{\frac{1}{2}} e^{-\frac{\varphi}{h}}\partial_{\nu}(-h^2\Delta)^{m-k-1}w \right\rVert_{L^2(\partial_{+}\Omega)}\\
			&\geq C_{\delta_0} h^{\frac{3}{2}+k}\left\lVert e^{-\frac{\varphi}{h}}\partial_{\nu}(-h^2\Delta)^{m-k-1}w \right\rVert_{L^2(\partial\Omega\setminus\Gamma_F)}.		\end{aligned}
	\end{equation*}
	Since $u-\wt{u} \in \mathcal{D}(\mathcal{L}(x,D))$, substituting $w=u-\wt{u}$, we obtain
	\begin{equation*}
		h^{2m}\left\lVert e^{-\frac{\varphi}{h}}\mathcal{L}(x,D)(u-\widetilde{u})\right\rVert_{L^2(\Omega)}
		\geq C_{\delta_0} h^{\frac{3}{2}+k}\left\lVert e^{-\frac{\varphi}{h}}\partial_{\nu}(-h^2\Delta)^{m-k-1}(u-\widetilde{u}) \right\rVert_{L^2(\partial\Omega\setminus\Gamma_F)}.
	\end{equation*}
	This gives, for all $0\leq k\leq m-1$,
	\begin{equation}\label{bdy_est_1}
		\left\lVert e^{-\frac{\varphi}{h}}\partial_{\nu}(-\Delta)^{m-k-1}(u-\widetilde{u}) \right\rVert_{L^2(\partial\Omega\setminus\Gamma_F)} \le \tilde{C}_{\delta_0} h ^{k+\frac{1}{2}}\left\lVert e^{-\frac{\varphi}{h}}\mathcal{L}(x,D)(u-\widetilde{u})\right\rVert_{L^2(\Omega)}.
	\end{equation}
Next we have
	\begin{equation}\label{eq_3}
		\begin{aligned}
			&\left| \int_{\partial\Omega\setminus\Gamma_F} \left[\partial_{\nu} (-\Delta)^{{m-k-1}}(u-\widetilde{u})\right] (-\Delta)^{k}\overline{v}\,\D S \right|\\
			&= \left| \int_{\partial\Omega\setminus\Gamma_F} e^{-\frac{\vp}{h}}\left[\partial_{\nu} (-\Delta)^{{m-k-1}}(u-\widetilde{u})\right] e^{\frac{\vp}{h}}(-\Delta)^{k}\overline{e^{\frac{-\vp+i\psi}{h}}\mathfrak{b}}\,\D S \right|\\
			&\leq C 
			\left\lVert e^{-\frac{\vp}{h}} \partial_{\nu} (-\Delta)^{{m-k-1}}(u-\widetilde{u})\right\rVert_{L^2(\PD_{+}\Omega)}
    			\left\lVert \left(-\Delta-\frac{1}{h}T\right)^k\overline{\mathfrak{b}}\right\rVert_{L^2(\partial\Omega)}.
		\end{aligned}
	\end{equation}
	Observe that $ \lVert\widetilde{r}\rVert_{H^{2m}_{\mathrm{scl}}(\Omega)} = \mathcal{O}(h^m) $. Thus using a similar argument as in \cite[Section 3]{Ghosh-Krishnan}, we have 
	\[\left\lVert\left(-\Delta-\frac{1}{h}T\right)^k\overline{\mathfrak{b}}\right\rVert_{L^2(\partial\Omega)} = \mathcal{O}\left(\frac{1}{h^k}\right). 
	\]
	This along with \eqref{bdy_est_1} and \eqref{eq_3} implies
	
	\begin{equation}\label{bdy_est_2}
		\begin{aligned}
			&\left| \int_{\partial\Omega\setminus\Gamma_F} \left[\partial_{\nu} (-\Delta)^{{m-k-1}}(u-\widetilde{u})\right] \Delta^{k}\overline{v}\,\D S \right|\\
			&\leq Ch^{\frac{1}{2}}\left\lVert e^{-\frac{\varphi}{h}}
			\mathcal{L}(x,D)(u-\widetilde{u})\right\rVert_{L^2(\Omega)}\\
			&\leq h^{\frac{1}{2}}\Bigg{(} \sum_{i_1\dots i_{m-2}=1}^{n} \,\left\| e^{-\frac{\varphi}{h}}\, (\widetilde{\overline{A}}^{\:m-2}_{i_1\dots i_{m-2}}-{\overline{A}^{\:m-2}_{i_1\dots i_{m-2}}})D^{m-2}_{i_1\dots i_{m-2}}\circ\Delta \widetilde{u}\right\|_{L^2(\Omega)}\\
            &\quad+\sum_{k=1}^{m-1}\sum_{i_1\dots i_k=1}^{n} \left\| e^{-\frac{\varphi}{h}} (\widetilde{A}^k_{i_1\dots i_k}-{A^k_{i_1\dots i_k}})D^k_{i_1\dots i_k} \widetilde{u}\right\|_{L^2(\Omega)} \\
            &\quad+\|e^{-\frac{\varphi}{h}}\,(\widetilde{q}-q)\widetilde{u}\|_{L^2(\Omega)}\Bigg{)}.
		\end{aligned}
	\end{equation}
	Note that
\begin{align}\label{eq_4}
& \left\lVert e^{-\frac{\vp}{h}} \left(\overline{A}^{\:m-2}_{i_1\dots  i_{m-2}} - \widetilde{\overline{A}}^{\:m-2}_{i_1\dots i_{m-2}}\right) D^{m-2}_{i_1\dots i_{m-2}}\circ \Delta\widetilde{u}\right\rVert_{L^2(\Omega)} 
\leq \mathcal{O}(h^{-(m-1)})\\
\intertext{and}
& \left\lVert e^{-\frac{\vp}{h}} \left(A^k_{i_1\dots   i_{k}} - \widetilde{A}^k_{i_1\dots i_k}\right) D^{k}_{i_1\dots i_k}\widetilde{u} \right\rVert_{L^2(\Omega)} 
\leq \mathcal{O}(h^{-k}), \qquad \mbox{for }k=1,\dots,m-1.
\end{align}
	Hence from \eqref{bdy_est_2}, for each $k$ with $0\leq k\leq m-1$, we finally obtain
	\begin{equation*}
		\left| \int_{\partial\Omega\setminus\Gamma_F} h^{m-1}\left[\partial_{\nu} (-\Delta)^{{m-k-1}}(u-\widetilde{u})\right] \Delta^k\overline{v}\,\D S \right|
		\leq \mathcal{O}(h^{\frac{1}{2}}) \to 0, 
		\qquad\mbox{as }h\to 0.
	\end{equation*}
	This proves \eqref{boundary_est_1.1}. 
	\end{proof}
\smallskip
\begin{lemma}\label{Prop_boundary_est_2}
	Let $u, \widetilde{u} \in H^{2m}(\Omega)$ solve \eqref{eq_1} with $\widetilde{u}$ as in the first equation of \eqref{cgos} and $v \in H^{2m}(\Omega)$ be as in the second equation of \eqref{cgos}. For $0<h\ll 1$ and $\Gamma_F \subset \partial\Omega$ be as before. For every indices $1\le i_1,i_2,\cdots,i_{m-2}\le n$, the following relations hold.
	\begin{align}
&	\lim\limits_{h\to 0} \int_{\partial\Omega}h^{m-1} \nu_{i_{m-1-k}}\left(D^{m-2-k}_{i_1 \dots i_{m-2-k}} \circ \Delta (u-\wt{u})\right) D^{k-1}_{i_{m-k} \dots i_{m-2}}\left(\overline{A}^{\:m-2}_{i_1 \dots i_{m-2}} \overline{v}\right) \,\D S =0,	\mbox{ for } 1\leq k\leq m-2 \label{eq_3.15},\\
&	\lim\limits_{h\to 0}  \int_{\partial\Omega}	h^{m-1} \left(\partial_{\nu}(u-\wt{u})\right)\, D^{m-2}_{i_{1} \dots i_{m-2}}\left(\overline{A}^{\:m-2}_{i_1 \dots i_{m-2}} \overline{v}\right) \,\D S=0 \label{eq_3.16},\\
&\lim\limits_{h\to 0}  \int_{\partial\Omega} h^{m-1} \nu_{i_{j-k+1}}\left(D^{j-k}_{i_1 \dots i_{j-k}}(u-\wt{u})\right) D^{k-1}_{i_{j-k+2} \dots i_{j}}\left(A^{j}_{i_1 \dots i_{j}} \overline{v}\right) \D S=0,\quad \forall 1\le j\le m-1, \quad  1\le k \le j.\label{eq_3.17}
\end{align}
	\end{lemma}
\begin{proof}
We start with the proof of \eqref{eq_3.15}. Note that
\begin{align*}
  &\left| \int_{\partial\Omega}\nu_{i_{m-1-k}}\left(D^{m-2-k}_{i_1 \dots i_{m-2-k}} \circ \Delta (u-\wt{u})\right) D^{k-1}_{i_{m-k} \dots i_{m-2}}\left(\overline{A}^{\:m-2}_{i_1 \dots i_{m-2}} \overline{v}\right) \,\D S\right|\\&\le \int_{\partial\Omega} \left|\nu_{i_{m-1-k}}\left(D^{m-2-k}_{i_1 \dots i_{m-2-k}} \circ \Delta (u-\wt{u})\right) D^{k-1}_{i_{m-k} \dots i_{m-2}}\left(\overline{A}^{\:m-2}_{i_1 \dots i_{m-2}} \overline{v}\right)\right| \,\D S\\
  &\le \left\lVert e^{-\frac{\vp}{h}}  D^{m-2-k}_{i_1 \dots i_{m-2-k}} \circ \Delta (u-\wt{u})  \right\rVert_{L^2(\PD\Omega)} \left\lVert e^{\frac{\vp}{h}} D^{k-1}_{i_{m-k} \dots i_{m-2}}\left(\overline{A}^{\:m-2}_{i_1 \dots i_{m-2}} \overline{v}\right) \right\rVert_{L^2(\PD\Omega)}.
\end{align*}
Since the tensor field $\overline{A}^{\:m-2}_{i_1 \dots i_{m-2}}\in C^{\infty}(\overline{\Omega})$,  we have 
$ \lVert D^{l} (\overline{A}^{\:m-2}_{i_1 \dots i_{m-2}})\rVert_{L^{\infty}(\PD\Omega)} \le C$ for all $0\le l\le k-1$. Also $v$ solves $\Lc^*(x,D)v = 0$ in $\Omega$ with $v$ given as in Proposition \ref{Prop_solvibility}. Now $ \|\wt{r}(x;h)\|_{H^{2m}_{\mathrm{scl}}(\Omega)} \leq ch^{m}$ gives $ \|D^{\alpha}\wt{r}\|_{L^2(\PD\Omega)} = \Oc (1) $ for $|\alpha|\le m-1$. Thus for $1 \leq k \leq m-2$ we have
\begin{align}\label{eq_3.21}
 \left\lVert e^{\frac{\vp}{h}} D^{k-1}_{i_{m-k} \dots i_{m-2}}\left(\overline{A}^{\:m-2}_{i_1 \dots i_{m-2}} \overline{v}\right) \right\rVert_{L^2(\PD\Omega)} = \Oc\left(\frac{1}{h^{k-1}}\right) \quad \mbox{for} \quad 1\le k \le m-2.   
\end{align}
Next we have 
\begin{align}\label{eq_3.22}
\notag \Bigg{|}\int_{\partial\Omega}\nu_{i_{m-1-k}}&\left(D^{m-2-k}_{i_1 \dots i_{m-2-k}} \circ \Delta (u-\wt{u}) \right) D^{k-1}_{i_{m-k} \dots i_{m-2}}\left(\overline{A}^{\:m-2}_{i_1 \dots i_{m-2}} \overline{v}\right)\D S\Bigg{|}\\
\notag &\leq \frac{C}{h^{k-1}}  \left\lVert e^{-\frac{\vp}{h}}  D^{m-2-k}_{i_1 \dots i_{m-2-k}} \circ \Delta (u-\wt{u}) \right\rVert_{L^2(\PD\Omega)},\\
\notag&\hspace{-.85in}\mbox{Using the classical trace theorem \cite[Section 5.5, Theorem 1]{Evans_pde_book}, we have}\\
\notag &\leq  \frac{C}{h^{k-1}}  \left\lVert e^{-\frac{\vp}{h}}  D^{m-2-k}_{i_1 \dots i_{m-2-k}} \circ \Delta (u-\wt{u})  \right\rVert_{H^1(\Omega)}\\
&\leq \frac{C}{h^k}\| e^{-\frac{\vp}{h}}D^{m-2-k}_{i_1 \dots i_{m-2-k}}  \Delta (u-\wt{u})\|_{L^2
		(\Omega)} + \frac{C}{h^{k-1}}\| e^{-\frac{\vp}{h}}D^{m-1-k}_{i_1 \dots i_{m-1-k}}  \Delta (u-\wt{u})\|_{L^2(\Omega)}.
\end{align}
For simplicity, we write, $w=(-h^2\Delta)(u-\wt{u})$ and estimate the RHS of \eqref{eq_3.22} with the help of \eqref{claim}: 
\begin{equation}\label{eq_3.25}
\begin{aligned}
\| e^{-\frac{\vp}{h}}D^{m-2-k}_{i_1 \dots i_{m-2-k}}  \Delta (u-\wt{u})\|_{L^2(\Omega)}
&=\frac{1}{h^{m-k}}\| e^{-\frac{\vp}{h}} h^{m-2-k}\,D^{m-2-k}_{i_1 \dots i_{m-2-k}} w\|_{L^2(\Omega)}\\
&\leq \frac{C}{h^{m-k}}	\left\| e^{-\frac{\vp}{h}}w\right\|_{H^{m-2-k}_{\mathrm{scl}}(\Omega)} \quad &&\mbox{using } \eqref{claim1}\\
&\leq \frac{C}{h^{m-k}}\,\sum_{j=0}^{[\frac{m-k}{2}]-1}\left\|e^{-\frac{\vp}{h}}(-h^2\Delta)^j w \right\|_{H^1_{\mathrm{scl}}(\Omega)} \quad &&\mbox{using } \eqref{claim}\\
&= \frac{C}{h^{m-k}}\,\sum_{j=1}^{[\frac{m-k}{2}]}\, \left\|e^{-\frac{\vp}{h}} (-h^2\Delta)^j (u-\wt{u}) \right\|_{H^1_{\mathrm{scl}}(\Omega)}\\
&\leq \frac{C}{h^{2m-k-1}}\,\sum_{j=1}^{[\frac{m-1}{2}]} h^{m-j}\,\left\|e^{-\frac{\vp}{h}}(-h^2\Delta)^j (u-\wt{u}) \right\|_{H^1_{\mathrm{scl}}(\Omega)}\\
&\leq \frac{Ch^{2m}}{h^{2m-k-1}}\left\lVert e^{-\frac{\vp}{h}}\mathcal{L}(x,D)(u-\wt{u})\right\rVert_{L^2(\Omega)}  \quad &&\mbox{using \eqref{bdy_carlemanestimate}}\\
&\leq Ch^{k+1} \left\lVert e^{-\frac{\vp}{h}}\mathcal{L}(x,D) (u-\wt{u})\right\rVert_{L^2(\Omega)} \mbox{ for all } 1\leq k \leq m-2.
\end{aligned}\end{equation}
Similarly we also have
\begin{align}\label{eq_3.26}
\|e^{-\frac{\vp}{h}}D^{m-1-k}_{i_1 \dots i_{m-1-k}}  \Delta (u-\wt{u})\|_{L^2(\Omega)} \le Ch^{k} \left\lVert e^{-\frac{\vp}{h}}\mathcal{L}(x,D)(u-\wt{u})\right\rVert_{L^2(\Omega)} \mbox{ for all } 1 \leq k \leq m-2. 
	\end{align}
Combining \eqref{eq_3.22}, \eqref{eq_3.25} and \eqref{eq_3.26}	we get
\begin{equation}\label{eq_3.27}
    \left| \int_{\partial\Omega}\nu_{i_{m-1-k}}\left(D^{m-2-k}_{i_1 \dots i_{m-2-k}} \circ \Delta (u-\wt{u})\right) D^{k-1}_{i_{m-k} \dots i_{m-2}}\left(\overline{A}^{\:m-2}_{i_1 \dots i_{m-2}} \overline{v}\right) \,\D S\right| \le Ch \left\lVert e^{-\frac{\vp}{h}}\mathcal{L}(x,D) (u-\wt{u})\right\rVert_{L^2(\Omega)}.
\end{equation}
	By \eqref{bdy_est_2} and \eqref{eq_4} we have 
\begin{equation}\label{bdy_est_2.2}
\left\lVert e^{-\frac{\vp}{h}}\mathcal{L}(x,D) (u-\wt{u})\right\rVert_{L^2(\Omega)} \le \frac{C_1}{h^{m-1}}.
\end{equation}
	This together with \eqref{eq_3.27}  implies
	\begin{align*}
		\lim\limits_{h\to 0}    \int_{\partial\Omega}\nu_{i_{m-1-k}} h^{m-1}\,\left(D^{m-2-k}_{i_1 \dots i_{m-2-k}} \circ \Delta (u-\wt{u})\right) D^{k-1}_{i_{m-k} \dots i_{m-2}}\left(\overline{A}^{\:m-2}_{i_1 \dots i_{m-2}} \overline{v}\right) \,\D S=0.
	\end{align*}
	This completes the proof of \eqref{eq_3.15}.
\smallskip 
	Now we sketch the proof of \eqref{eq_3.16} and \eqref{eq_3.17}.
	In order to prove \eqref{eq_3.16} we closely follow the arguments used in Lemma \ref{Prop_boundary_est_1}. Consider
	\begin{align}\label{eq_3.28}
	   & \left| \int_{\partial\Omega} \left(\partial_{\nu}(u-\wt{u})\right)\, D^{m-2}_{i_{1} \dots i_{m-2}}\left(\overline{A}^{\:m-2}_{i_1 \dots i_{m-2}} \overline{v}\right)\D S \right|\nonumber\\& = \left| \int_{\partial\Omega\setminus \Gamma_F} \left(\partial_{\nu}(u-\wt{u})\right)\, D^{m-2}_{i_{1} \dots i_{m-2}}\left(\overline{A}^{\:m-2}_{i_1 \dots i_{m-2}} \overline{v}\right)\D S \right|\nonumber\\
	    &\le \| e^{-\frac{\vp}{h}} \left(\partial_{\nu}(u-\wt{u})\right)\,  \|_{L^2(\partial\Omega\setminus\Gamma_F)} \| e^{\frac{\vp}{h}}\, D^{m-2}_{i_{1} \dots i_{m-2}}\left(\overline{A}^{\:m-2}_{i_1 \dots i_{m-2}} \overline{v}\right)\|_{L^2(\partial\Omega\setminus\Gamma_F)}\nonumber \\
	    &\le \tilde{C}_{\delta_0} h ^{m-\frac{1}{2}}\left\lVert e^{-\frac{\varphi}{h}}\mathcal{L}(x,D)(u-\widetilde{u})\right\rVert_{L^2(\Omega)}\| e^{\frac{\vp}{h}}\, D^{m-2}_{i_{1} \dots i_{m-2}}\left(\overline{A}^{\:m-2}_{i_1 \dots i_{m-2}} \overline{v}\right)\|_{L^2(\partial\Omega)}.
	\end{align}
	The last inequality above follows from \eqref{bdy_est_1} and $ \tilde{C}_{\delta_0}$ is  same as in the proof of Lemma \ref{Prop_boundary_est_1}. A similar line argument used in  \eqref{eq_3.21} gives
	\begin{align*}
	    \| e^{\frac{\vp}{h}}\, D^{m-2}_{i_{1} \dots i_{m-2}}\left(\overline{A}^{\:m-2}_{i_1 \dots i_{m-2}} \overline{v}\right)\|_{L^2(\partial\Omega)} = \Oc\left(\frac{1}{h^{m-2}}\right).
	\end{align*}
	This together with \eqref{eq_3.28} gives
	\begin{equation}
	  \left| \int_{\partial\Omega} \left(\partial_{\nu}(u-\wt{u})\right)\, D^{m-2}_{i_{1} \dots i_{m-2}}\left(\overline{A}^{\:m-2}_{i_1 \dots i_{m-2}} \overline{v}\right) \D S\right| \le C h^{\frac{3}{2}}   \left\lVert e^{-\frac{\varphi}{h}}\mathcal{L}(x,D)(u-\widetilde{u})\right\rVert_{L^2(\Omega)}.
	\end{equation}
	We now arrive in the same situation similar to \eqref{eq_3.27}. This completes the proof of \eqref{eq_3.16}. The proof of \eqref{eq_3.17} follows similarly as in the proof of \eqref{eq_3.15}. However  a similar line of analysis shows that for $1 \leq k \leq j$ we have
	\begin{align*}
	  &  \left| \int_{\partial\Omega}  \nu_{i_{j-k+1}}\left(D^{j-k}_{i_1 \dots i_{j-k}}(u-\wt{u})\right) D^{k-1}_{i_{j-k+2} \dots i_{j}}\left(A^{j}_{i_1 \dots i_{j}} \overline{v}\right) \D S \right|\nonumber\\
	    &\le C h^{m-j} \left\lVert e^{-\frac{\varphi}{h}}\mathcal{L}(x,D)(u-\widetilde{u})\right\rVert_{L^2(\Omega)} \quad \mbox{for}  \quad 1\le j\le m-1.
	\end{align*}
	This completes the proof of Lemma \ref{Prop_boundary_est_2}.
\end{proof}

\bigskip

\smallskip

\begin{remark}\label{Rem_isotropy_2}
    Note that, in \eqref{eq_3.25} and \eqref{eq_3.26} we use the fact that $A^m=i_{\delta}\overline{A}^{\: m-2}$ for some symmetric $m-2$ tensor $\overline{A}^{\: m-2}$. For a general symmetric $m$ tensor $A^m$ we would obtain an estimate of order $\Oc(h^k)$ in \eqref{eq_3.25} and $\Oc(h^{k-1})$ in \eqref{eq_3.26}, which will not give us the desired decay when we combine them with \eqref{eq_3.22}. 
\end{remark}

Let us define the notion of $l$-isotropy of a symmetric $k$-tensor, where $l\leq [k/2]$ is a non-negative integer.
\begin{definition}\label{def_isotropy}
Let $2\leq k \leq m$ and $0\leq l\leq [\frac{k}{2}]$ be integers. Let $W=\lb W_{i_1\cdots i_k}\rb$ be a symmetric $k$-tensor. We say $W$ is $l$-isotropic if $l$ is the largest integer for which there is a symmetric $k-2l$ tensor $\Wc$ such that $W$ can be expressed as
\begin{equation}\label{l-isotropic}
W_{i_1\dots i_k} = i^{l}_{\d} \Wc
:= \sigma(i_1\dots i_k) \Wc_{i_1\dots i_{k-2l}}\d_{i_{k-2l+1}i_{k-2l+2}}\dots \d_{i_{k-1}i_{k}},
\end{equation}
where $\d_{ij}$ is the Kronecker delta symbol.
\end{definition}
Note that if $W$ is $0$-isotropic, then $\Wc = W$ in the notation of the above definition.
Observe that if some of the terms $(\widetilde{A}^k-A^k)$ are isotropic, then  we can in fact improve the estimates in Lemmas \ref{Prop_boundary_est_1} and \ref{Prop_boundary_est_2}.
Let us denote $W^k = A^k-\widetilde{A}^k$ for $k=1,\dots,m$, $W^0 = q-\widetilde{q}$. Additionally, since $A^{m}=i_{\delta}\overline{A}^{m-2}$ (respectively for $\wt{A}^{m}$), we denote $\overline{W}^{\:m-2} = \overline{A}^{\:m-2} - \widetilde{\overline{A}}^{\:m-2}$ in $\Omega$.
\begin{lemma}\label{Prop_boundary_est_3}
Let $m\geq 3$.
If $W^k$ is $l_k$-isotropic for some $l_k\in \{ 0,1,\dots,[\frac{k}{2}]\}$ for $k=2,\cdots,m$, then
\begin{equation}\label{boundary_est_1.2}
	\lim_{h\to 0} \int_{\partial\Omega\setminus\Gamma_F} h^{\gamma}\left[\partial_{\nu} (-\Delta)^{{m-k-1}}(u-\widetilde{u})\right] \Delta^{k}\overline{v}\,\D S = 0 \quad\mbox{for} \quad 0\leq k \leq (m-1),
\end{equation}
where $\gamma = \mathrm{max}\{(k-l_k) \st 2\leq k\leq m\}$ and $u,\widetilde{u},v$ be as in Lemma \ref{Prop_boundary_est_1}.

Moreover, with the same assumptions,
\begin{align}
&\lim\limits_{h\to 0} \int_{\partial\Omega}h^{\gamma} \nu_{i_{m-1-k}}\left(D^{m-2-k}_{i_1 \dots i_{m-2-k}} \circ \Delta (u-\wt{u})\right) D^{k-1}_{i_{m-k} \dots i_{m-2}}\left(\overline{A}^{\:m-2}_{i_1 \dots i_{m-2}} \overline{v}\right) \,\D S =0,\quad  \mbox{ for } 1\leq k \leq m-2 \label{eq_3.29}\\
&\lim\limits_{h\to 0}  \int_{\partial\Omega}	h^{\gamma} \left(\partial_{\nu}(u-\wt{u})\right)\, D^{m-2}_{i_{1} \dots i_{m-2}}\left(\overline{A}^{\:m-2}_{i_1 \dots i_{m-2}} \overline{v}\right) \,\D S=0 \label{eq_3.30}\\
&\lim\limits_{h\to 0}  \int_{\partial\Omega} h^{\gamma} \nu_{i_{j-k+1}}\left(D^{j-k}_{i_1 \dots i_{j-k}}(u-\wt{u})\right) D^{k-1}_{i_{j-k+2} \dots i_{j}}\left(A^{j}_{i_1 \dots i_{j}} \overline{v}\right) \D S=0 \label{eq_3.31}, \quad \mbox{for }1\leq k \leq m-1.
\end{align}
\end{lemma}
\begin{proof}
First let us note that $l_m\geq 1$ due to the form of $A^m$ and $\widetilde{A}^m$ in \eqref{operator}. We have 
\begin{equation}\label{bdy_est_4}
	\begin{aligned}
	\left\lVert e^{-\frac{\varphi}{h}}
	\mathcal{L}(x,D)(u-\widetilde{u})\right\rVert_{L^2(\Omega)}
	&\leq 
    \sum_{k=2}^{m}\sum_{i_1\dots i_{k-2l_k}=1}^{n} \left\| e^{-\frac{\varphi}{h}} \Wc^k_{i_1\dots i_{k-2l_k}}D^{k-2l_k}_{i_1\dots i_{k-2l_k}}\circ\Delta^{l_k} \widetilde{u}\right\|_{L^2(\Omega)} \\
    &\quad+\sum_{i=1}^{n} \left\| e^{-\frac{\varphi}{h}} W^1_i D^{1}_{i} \widetilde{u}\right\|_{L^2(\Omega)} \\
    &\quad+\|e^{-\frac{\varphi}{h}}\,(\widetilde{q}-q)\widetilde{u}\|_{L^2(\Omega)},
	\end{aligned}
\end{equation}
where $W^k_{i_1\dots i_k} = i^{l_k}_{\delta}\Wc^k_{i_1\dots i_{k-2l_k}}$ for $k=2,\dots, m$, according to the notation in \eqref{l-isotropic}.
Observe that
\begin{equation*}
\begin{aligned}
\left\lVert e^{-\frac{\vp}{h}} \Wc^k_{i_1\dots i_{k-2l_k}}  D^{k-2l_k}_{i_1\dots i_{k-2l_k}}\circ \Delta^{l_k} \widetilde{u}\right\rVert_{L^2(\Omega)}
\leq& \mathcal{O}(h^{-(k-l_k)}), \quad \mbox{for }k=2,\dots,m\\
\left\lVert e^{-\frac{\vp}{h}} W^1_i D^{1}_{i}\widetilde{u} \right\rVert_{L^2(\Omega)} 
+ \left\lVert e^{-\frac{\vp}{h}} \left(q - \widetilde{q}\right) \widetilde{u} \right\rVert_{L^2(\Omega)} 
\leq& \mathcal{O}(h^{-1}).
\end{aligned}
\end{equation*}
Observe that if $\g = \max \{(k-l_k)\st k=2,\dots,m\}$ then $h^{-(k-l_k)} \leq h^{-\g}$ for $0<h<1$ and thus combining the above estimates with \eqref{bdy_est_4} we obtain
\begin{equation}\label{bdy_est_4.1}
\left\lVert e^{-\frac{\varphi}{h}} \mathcal{L}(x,D)(u-\widetilde{u}) \right\rVert_{L^2(\Omega)}
\leq \Oc(h^{-\gamma}).
\end{equation}
Now substituting \eqref{bdy_est_4.1} in  \eqref{bdy_est_2} and in \eqref{bdy_est_2.2} we complete the proof.
\end{proof}

\begin{remark}
Observe that, when $m=2$ from \eqref{boundary_term} we have
\begin{equation}\label{boundary_terms_m2}
\begin{aligned}
\Bc =-\sum_{k=0}^{1}\int_{\partial\Omega\setminus \Gamma_{F}} \left(\partial_{\nu} (-\Delta)^{2-k-1}(u-\wt{u})\right) \overline{\left((-\Delta)^{k}v\right)}\,\D S
-\int_{\partial\Omega} \partial_{\nu}(u-\wt{u})\left(A^{2} \overline{v}\right) \,\D S.
\end{aligned}
\end{equation}
Therefore, we obtain $\lim_{h \to 0} h\Bc = 0$ by using Lemma \ref{Prop_boundary_est_1} and \eqref{eq_3.16} of Lemma \ref{Prop_boundary_est_2}.
\end{remark}

\subsection{Coefficient Determination}\label{Section_coeff_deter}

We describe an inductive procedure to show the unique determination of the coefficients, proving our main result, Theorem \ref{mainresult_1}.
\begin{proof}[Proof of Theorem \ref{mainresult_1}]
We give the proof by induction.
Recall the notation $W^{k}$ given above Lemma \ref{Prop_boundary_est_3}. 

We multiply \eqref{Int_id} by $h^{m-1}$ and let $h\to 0$. Then from \eqref{Int_id}, using Lemmas \ref{Prop_boundary_est_1} and \ref{Prop_boundary_est_2},  we have 
    \Beq\label{Step0-Eq0}
    \sum\limits_{i_1\cdots i_{m-2}=1}^{n}\int \overline{W}^{\:m-2}_{i_1\cdots i_{m-2}}\prod\limits_{k=1}^{m-2} D_{i_k}\lb \vp + \I \psi\rb (-T a_0) \overline{b}_0 \, \D x+ \sum\limits_{i_1\cdots i_{m-1}=1}^{n}\int W^{m-1}_{i_1\cdots i_{m-1}} \prod\limits_{k=1}^{m-1} D_{i_k}\lb \vp + \I \psi\rb a_0 \overline{b}_0 \, \D x=0.
    \Eeq
Let us choose $a_0(x)= h(z)(z-\overline{z})^{\frac{2-n}{2}}$ and $\overline{b}_0(x)= g(\theta)(z-\overline{z})^{\frac{2\A-n}{2}}$, $1\leq \A\leq m$. We have $T a_0(x)=0$. Also 
for any smooth function $g(\theta)$ and any holomorphic function $h(z)$,
\Beq\label{Step-0-Eq2}
\sum_{i_1,\dots,i_{m-1}=1}^{n} \int_{\Omega} W^{m-1}_{i_1\dots i_{m-1}} \left(\prod_{j=1}^{m-1}D_{i_j}(\vp + \I\psi)\right)(z-\overline{z})^{\A+1-n} h(z)g(\theta) \, \D x=0,\quad 1\leq \A \leq m.
\Eeq
Next we follow the approach of \cite{K_S}. This will reduce the above integral identity to a sum of MRT on the unit sphere bundle.
Transforming \eqref{Step-0-Eq2} in cylindrical polar coordinates as in \eqref{cylindrical_coordinates}, for each $1\leq \A \leq m$, the above identity gives
\begin{equation*}
\sum_{i_1,\dots,i_{m-1}=1}^{n} \int_{\Omega} W^{m-1}_{i_1\dots i_{m-1}} \left(\prod_{j=1}^{m-1}D_{i_j}(\vp + \I\psi)\right)r^{\A+1-n}h(z)g(\theta)r^{n-2} \, \D x_1\D r \D\theta=0.
\end{equation*}
Varying $g$ for almost every $\theta \in \Sb^{n-2}$ we get
\begin{equation}\label{Step0-Eq3}
\int_{\Omega_{\theta}} \left(W^{m-1}_{i_1\dots i_{m-1}} (-\I)^{m-1}
\sum_{i_1,\dots,i_{m-1}=1}^{n} \prod_{j=1}^{m-1}(e_1+\I e_r)_{i_{j}}\right)r^{\A-1}h(z) \, \D x_1\D r=0, \quad 1\leq \A \leq m.
\end{equation}
Next we choose 
\[h(z) =
e^{-\I\lambda z} = e^{\lambda(r-\I x_1)},\]
and for each $0\leq \A \leq m-1$ we obtain 
\[ \sum_{i_1,\dots,i_{m-1}=1}^{n} \int_{\Omega_{\theta}} r^{\A}e^{\lambda r} \left(W^{m-1}_{i_1\cdots i_{m-1}} \prod_{j=1}^{m-1}(e_1+\I e_r)_{i_{j}}\right) e^{-\I \lambda x_1} \, \D x_1\D r=0, \quad \mbox{for a.e. }\theta \in \Sb^{n-2}.
\]
Note that we can extend $W^{m-1}(\cdot,\theta)=0$ over $\Rb^2 \setminus \Omega_{\theta}$ and perform the integration in the $x_1$ variable to obtain
\begin{equation}\label{Step0-Eq4}
\sum_{i_1,\dots,i_{m-1}=1}^{n} \int_{\Rb} r^{\A}e^{\lambda r} \left(\widehat{W}^{m-1}_{i_1\cdots i_{m-1}}(\lambda,r,\theta) \prod_{j=1}^{m-1}(e_1+\I e_r)_{i_{j}}\right) \, \D r=0, \quad \mbox{for a.e. }\theta \in \Sb^{n-2},\mbox{ and for all } \lambda \in \Rb,
\end{equation}
where we denote $\widehat{\cdot}$ to be the partial Fourier transform in the $x_1$ variable.

We vary the center of the polar coordinate to be any point $x_0 \in \Rb^{n-1}$ such that $(0,x_0) \notin  \overline{\Omega}$ as described in \eqref{cylindrical_coordinates}. Thus varying the center $x_0$ we can have the above integral identity along any line on all the hyperplanes perpendicular to $e_1=(1,0,\cdots,0)$.
We set $\lambda=0$ and using Lemma \ref{kernel_mrt_sphere_bundle}, shown below, in \eqref{Step0-Eq4}, we see that there exists a symmetric tensor field $\wh{F}^{m-1,1,0}$ such that 
\[
\widehat{W}^{m-1}(0,r,\theta) = 
i_{\delta}\wh{F}^{m-1,1,0}(0,r,\theta), \quad \mbox{for a.e. }r,\theta.
\]
Here $\wh{F}^{m-1,1,0}$ is a symmetric $m-3$ tensor if $m\geq 3$ and $\wh{F}^{m-1,1,0}=0$ otherwise.
Next, we differentiate \eqref{Step0-Eq4} with respect to $\lambda$ and set $\lambda=0$.  We get 
\Beq\label{Step0-Eq5}
\int r^{\A+1} \left(\widehat{W}^{m-1}_{i_1\cdots i_{m-1}}(0,r,\theta) \prod_{j=1}^{m-1}(e_1+\I e_r)_{i_{j}}\right) \D r + \int r^{\A} \frac{\D}{\D \lambda} \left(\widehat{W}^{m-1}_{i_1\cdots i_{m-1}}(0,r,\theta) \prod_{j=1}^{m-1}(e_1+\I e_r)_{i_{j}}\right) \D r=0.
\Eeq
Since $\wh{W}^{m-1}(0,r,\theta) = i_{\delta} \wh{F}^{m-1,1,0}(0,r,\theta)$, we have that the first term in \eqref{Step0-Eq5} vanishes using the fact that $(e_1+\I e_r)\cdot (e_1+\I e_r)=0$. Now repeating the same steps as before, we see that there exists a symmetric $m-3$ tensor field $\wh{F}^{m-1,1,1}$ such that 
\[
\frac{\D}{\D \lambda}|_{\lambda=0} \wh{W}^{m-1}(0,r,\theta) = i_{\delta} \wh{F}^{m-1,1,1}(0,r,\theta), \quad \mbox{ for a.e. } r, \theta.
\]
Continuing the same argument, we get that for all $\g = 0,1,\cdots$, there exist symmetric $m-3$ tensor fields $\wh{F}^{m-1,1,\g}$ such that 
\begin{equation}\label{Step0-Eq6}
\frac{\D ^\g}{\D \lambda^\g}|_{\lambda=0}\widehat{W}^{m-1}(0,r,\theta) = i_{\delta}\wh{F}^{m-1,1,\g}(0,r,\theta), \quad \mbox{for a.e. }r,\theta. 
\end{equation}
Since $W^{m-1}$ is supported in $\overline{\Omega}$, we have that $\widehat{W}^{m-1}$ is analytic in $\lambda$. Using Payley-Weiner theorem, we have that   $\widehat{W}^{m-1}(\lambda,r,\theta)=i_{\delta}\wh{F}^{m-1,1}$ a.e. in $\Omega$, where $\wh{F}^{m-1,1}(\lambda,r,\theta)$ is a symmetric $m-3$ tensor  field. 
Taking the inverse Fourier transform in the sense of tempered distributions, and noting that the inverse Fourier transform commutes with $i_{\delta}$, we get that there exists a symmetric $m-3$ tensor field  $F^{m-1,1}$ if $m\geq 3$ and $0$ otherwise, such that 
\[
W^{m-1}(x_1,r,\theta)=i_{\delta} F^{m-1,1}(x_1,r,\theta).
\]

Finally, note that using this in \eqref{Step0-Eq0}, we have that the second integral is identically $0$ and therefore we get 
\[
\sum\limits_{i_1\cdots i_{m-2}=1}^{n}\int \overline{W}^{\:m-2}_{i_1\cdots i_{m-2}}\prod\limits_{k=1}^{m-2} D_{i_k}\lb \vp + \I \psi\rb (-T a_0) \overline{b}_0 \, \D x=0.
\]
Now let us choose 
\[
a_0=(z-\overline{z})^{\frac{4-n}{2}} h(z),
\]
where $h(z)$ is any holomorphic function. A simple calculation shows that 
\[
T a_0 = -(z-\overline{z})^{\frac{2-n}{2}} h(z).
\]
Then choosing $b_0=(z-\overline{z})^{\frac{2\A-n}{2}}g(\theta)$ for any smooth function $g(\theta)$ and for all $1\leq \A\leq m$, we are in the same set-up as before. Repeating the same steps, we get that there is a symmetric $m-4$ tensor field $F^{m,2}$ such that
\[
\overline{W}^{\:m-2}(x_1,r,\theta)= i_{\delta} F^{m,2}(x_1,r,\theta).
\]
Since $W^m = i_{\delta} \overline{W}^{\:m-2}$, to conclude the initial step of induction, we have shown that 
\[
W^{m}= i^2_{\delta} F^{m,2} \mbox{ and } W^{m-1}= i_{\delta}F^{m-1,1}.
\]

Assume by induction, for some $k\geq 1$, that 
\begin{align*}
    W^{m}=i^{k+1}_{\delta} F^{m,k+1}, W^{m-1}= i^{k}_{\delta}F^{m-1,k}, W^{m-2}=i_{\delta}^{k-1}F^{m-2,k-1}, 
    \cdots, W^{m-k} = i_{\delta} F^{m-k,1}.
\end{align*}
for symmetric tensor fields $F^{m,k+1},\cdots, F^{m-k,1}$ of appropriate ranks.

We next show for the case $m-k-1$. By the induction assumption, multiplying \eqref{Int_id} with $h^{m-k-1}$ and letting $h\to 0$, combined with Lemma \ref{Prop_boundary_est_3}, we have that 
\begin{equation}\label{Step2-Eq0}
\begin{aligned}
0=\sum_{l=0}^{k+1}\sum_{i_1,\cdots,i_{m-2k-2+l}=1}^{n} \int_{\Omega} F^{m-l,k+1-l}_{i_1\dots i_{m-2k-2+l}} \left(\prod_{j=1}^{m-2k-2+l}D_{i_j}(\vp + \I\psi)\right)\left[(-T)^{k+1-l}a_0(x)\right]\, \overline{b_0(x)} \, \D x.
\end{aligned}
\end{equation}
Let 
\begin{equation}\label{amplitudes}
\begin{aligned}
a_0(x) = h(z)(z-\overline{z})^{(2(k+2)-2l-n)/2} \quad \mbox{and}\quad 
\overline{b_0(x}) = g(\theta)(z-\overline{z})^{(2\A-n)/2}, \quad 1\leq \A \leq m,
\end{aligned}
\end{equation}
with $l$ starting from $0$ to $k+1$ in the descending order. Here $g(\cdot)$ is any smooth function in $\theta$ variable, $z=x_1+\I r \in \Omega_{\theta}$ as in the notation defined in \eqref{cylindrical_coordinates}, \eqref{Omega_theta} and $h$ is a holomorphic function in $z \in \Omega_{\theta}$.
Observe that
\[  T^{k+1-l}a_0(x)= \mh(z)(z-\overline{z})^{(2-n)/2} \quad \mbox{and}\quad T^{l_1}a_0(x)=0, \quad \mbox{whenever } l_1>k+1-l.
\]
Thus with the choices of $a_0$ and $b_0$, we sequentially get, starting from $l=0$ to $k+1$ in descending order, that 

\begin{equation}\label{Step-2-Eq1}
\sum_{i_1,\cdots,i_{m-2k-2+l}=1}^{n} \int_{\Omega} F^{m-l,k+1-l}_{i_1\dots i_{m-2k-2+l}} \left(\prod_{j=1}^{m-2k-2+l}D_{i_j}(\vp + \I\psi)\right)(z-\overline{z})^{\A+1-n} h_1(z) g(\theta) \D x=0,\quad 1\leq \A \leq m. 
\end{equation}
Next proceeding exactly as in the initial induction step, we then get that 

\begin{align*}
   & W^{m}=i^{k+2}_{\delta} F^{m,k+2}, W^{m-1}= i^{k+1}_{\delta}F^{m-1,k+1}, W^{m-2}=i_{\delta}^{k}F^{m-2,k},\\
   &W^{m-3}=i_{\delta}^{k-1} F^{m-3,k-1},\cdots, W^{m-k} = i_{\delta}^2 F^{m-k,2}, W^{m-k-1}=i_{\delta}F^{m-k-1,1}
\end{align*}
for symmetric tensor fields $F^{m,k+2},\cdots, F^{m-k-1,1}$ of appropriate ranks. The proof is complete by induction.

The aforementioned inductive procedure shows that in a finite number of steps we will arrive at the ray transform of a function in the case of an even order tensor field and the MRT of a sum of a function and a vector field in the case of odd order. In either case, we can conclude that these coefficients are $0$. This concludes the proof of our main result.
\end{proof}

The following lemma was used in the proof of the main theorem above, which we prove now.

\begin{lemma}\label{kernel_mrt_sphere_bundle}
    Let $f(x_1,x')$ be a symmetric $m$-tensor field in $\Rn$, compactly supported in $x'=(r,\theta)$ variable. Assume that 
    \begin{align}\label{intergral_identities}
       \int_{\Rb} r^{\A}\, f_{i_1\cdots i_m}(0,r,\theta) \left(\prod_{j=1}^{m}(e_1+\I e_r)_{i_j}\right) \D r& = 0, \quad \mbox{for a.e. }\theta \in \Sb^{n-2},
    \end{align}
 for each $ 0\le \A \le m$. Then, for $m\ge 2$
\begin{equation*}
    f(0,r,\theta)=i_{\d} v(0,r,\theta),
\end{equation*}
for some symmetric $m-2$ tensor field $v$ compactly supported in $(r,\theta)$ variable and \begin{align*}
     f(0,r,\theta)=0 \quad \mbox{if} \quad m=0,1.
\end{align*}
Moreover for $m\ge 2$, $v$ can be chosen as follows:
 \begin{align}\label{existence_of_v}
        v_{i_1\cdots i_p 1\cdots 1}&=  \binom{m-2}{m-p-2}^{-1}\left( \sum\limits_{l=1}^{[\frac{p+2}{2}]}(-1)^{l-1}\, i_{\d}^{l-1} \tilde{f}^{p+2-2l}\right), 
\end{align}
where $\tilde{f}^{p+2-2l} = \binom{m}{p+2-2l} f_{i_1\cdots i_{p+2-2l}1\cdots 1} $ a symmetric  tensor field of order $p+2-2l$.
\end{lemma}
\begin{proof}
Let $f$ be a symmetric $m$ tensor field, $m\geq 2$. Define two operators $i_{\d}$  and $j_{\d}$ as follows \cite{Sharafutdinov_book}:  
\begin{align}
 (i_{\d} f)_{i_1\cdots i_{m+2}}&=  \sigma(i_1\dots i_{m+2}) f_{i_1\dots i_{m}}\d_{i_{m+1}i_{m+2}}, \label{definition_i}\\
 (j_{\d}f)_{i_1\cdots i_{m-2}}&= \sum\limits_{k=1}^n f_{i_1\cdots i_{m-2} kk} \label{definition_j}.
\end{align}
The operators $i_{\d}$ and $j_{\d}$ are dual to each other with respect to the standard inner product on symmetric tensors. From \cite[Lemma 2.3] {Dairbekov-Sharafutdinov}, we have the following orthogonal decomposition for $f$ as 
$f =g +i_{\d} v$ with $j_{\d} g=0$. 
Note that since $f$ is compactly supported, $v$ is compactly supported as well. 
Note that $i_{\delta} v(0,r,\theta)$ lies in the kernel of \eqref{intergral_identities}. Using the above decomposition, we rewrite \eqref{intergral_identities} as 
\begin{align*}
\int_{\Rb} r^{\A}\, \left(g_{i_1\dots i_{m}}(0,r,\theta) \prod_{j=1}^{m}(e_1+\I e_r)_{i_{j}}\right)\, \D r&=0 \quad \mbox{for a.e. }\theta \in \Sb^{n-2},
\end{align*}
and for $0\le \alpha \le m$. Rewriting the above equation, we have,
\begin{align}\label{main_identity}
    \int_{\Rb} r^{\A}\, g(0,r,\theta) \, \D r& = 0 \quad \mbox{for} \quad 0\le \alpha \le m,
\end{align}
where we denote $g$ to be the sum of symmetric tensors in $i_j=2,\cdots,n$ indices for $j=1,\dots,m$ as follows: 
\begin{equation}\label{MRT_1}
g= \sum_{p=0}^{m} \I^p \tilde{g}^{p}_{i_1\cdots i_p}\, (e_r)_{i_1}\cdots (e_r)_{i_p}\quad
\mbox{with}\quad 
\tilde{g}^{p}_{i_1 \cdots i_p} =  c_p
g_{i_1\cdots i_{p} 1 \cdots 1},
\end{equation}
where $ c_p= \binom{m}{p}$.
Let $m$ be even with $m=2k$.  The odd case follows by  a very similar argument. Using Theorem \eqref{Th_MRT_unit_disk} we have that (denoting $x'=(r,\theta)$)
\begin{align}\label{even_and_odd_highest_component}
   \tilde{g}^{2k} (0,x') 
   =\left( \sum\limits_{l=1}^{k} i^{l}_{\d} (-1)^{l+1}\tilde{g}^{2k-2l}(0,x')\right),\quad  \tilde{g}^{2k-1}(0,x')= \left( \sum\limits_{l=1}^{k-1} i^{l}_{\d} (-1)^{l+1}\tilde{g}^{2k-2l-1}(0,x')\right)
\end{align}
Since  $j_{\d} g=0 $, we see
\begin{align*}
    (j_{\d} \tilde{g}^{2k})_{i_1\cdots i_{2k-2}} 
    = \sum_{l=2}^{n}  \tilde{g}^{2p}_{i_1\cdots i_{2k-2} ll} 
    = - \frac{1}{c_{2k-2}} \tilde{g}^{2k}_{i_1\cdots i_{2k-2} 1 1}  = - \frac{1}{c_{2k-2}} \tilde{g}^{2k-2}_{i_1\cdots i_{2k-2}}.
\end{align*}
Note that, since $\tilde{g}^{2k}$ is a tensor in $i_j=2,\dots,n$ indices for $1\leq j \leq 2k$, therefore, the sum in the definition of $j_{\d}$ is over $2$ to $n$.
Continuing in this way we obtain
\begin{align*}
     j^{l}_{\d}\tilde{g}^{2k}= \frac{(-1)^l}{c_{2k-2l}} \tilde{g}^{2k-2l}= (-1)^l\, d_{2k-2l}\,\tilde{g}^{2k-2l},
\end{align*}
where $d_{2k-2l} = \frac{1}{c_{2k-2l}}>0 $ for $ 0\le l \le k$.
This implies
\begin{align*}
    \langle \tilde{g}^{2k}, \tilde{g}^{2k}\rangle
    &= \left\langle \tilde{g}^{2k},\left( \sum\limits_{l=1}^{k} i^{l}_{\d} (-1)^{l+1}\tilde{g}^{2k-2l}(x')\right) \right \rangle\\
    &= \sum\limits_{l=1}^{k} \left\langle j^{l}_{\d} \tilde{g}^{2k}, (-1)^{l+1}\tilde{g}^{2k-2l} \right \rangle.\\
    &= \sum\limits_{l=1}^{k} \left\langle  d_{2k-2l}\, (-1)^l  \tilde{g}^{2k-2l}, (-1)^{l+1}\tilde{g}^{2k-2l} \right \rangle\\
    &= - \sum\limits_{l=1}^{k}d_{2k-2l}  \left\langle \tilde{g}^{2k-2l}, \tilde{g}^{2k-2l} \right \rangle.
\end{align*}
This implies
\begin{align}\label{plancheral_type_relation}
     \langle \tilde{g}^{2k}, \tilde{g}^{2k}\rangle +  \sum\limits_{l=1}^{k}d_{2k-2l}  \left\langle \tilde{g}^{2k-2l}, \tilde{g}^{2k-2l} \right \rangle=0.
\end{align}
    This gives $\tilde{g}^{2l}=0 $ for $0\le l \le k$. Using a  similar analysis we can show that $ \tilde{g}^{2l-1}=0 $ for $1\le l \le k$. This implies $ g_{i_1 \cdots i_m}(0,x')=0$ for $ 1\le i_1,\cdots,i_m \le n$. Hence we have $f=i_\delta v$ for a compactly supported $m-2$ tensor field $v$ concluding the proof.
    
    We now give an explicit construction of $v$ by expressing it in terms of $f$ as follows:
    \begin{align}\label{def_v}
        v_{i_1\cdots i_p 1\cdots 1}&=  \binom{m-2}{m-p-2}^{-1}\left( \sum\limits_{l=1}^{[\frac{p+2}{2}]}(-1)^{l-1}\, i_{\d}^{l-1} \tilde{f}^{p+2-2l}\right), 
\end{align}
for $0\le p \le m-2$ and $2\le i_1,\cdots ,i_{p} \le n$. In the above definition of $v$, $1$ appears $m-p-2$ times. We now show that $ f(0,x') = i_{\delta} v(0,x')$. 
   Assume that $ i_{j+1}, \cdots, i_m =1$, $ 2 \le i_{1},\cdots, i_j \le n$ and $m-j\ge 2.$ Then by definition we can write 
$(i_{\delta}v)_{i_1\cdots i_{j}1\cdots 1}$ as 
\begin{align*}
(i_{\delta}v)_{i_1\cdots i_{j}1\cdots 1}
&= \binom{m}{2}^{-1}\left(\binom{m-j}{2} v_{i_1\cdots i_{j}1\cdots 1}\delta_{11}
+ \binom{j}{2} \sigma(i_1\cdots i_{j}) v_{i_1\cdots i_{j-2}1\cdots 1} \otimes \delta_{i_{j-1}i_{j}} \right).
\end{align*}
    This together with \eqref{def_v} gives
\begin{align*}
   \binom{m}{2} (i_{\delta}v)_{i_1\cdots i_{j}1\cdots 1}=&  \binom{m-j}{2}  \binom{m-2}{m-j-2}^{-1}\left( \sum\limits_{l=1}^{[\frac{j+2}{2}]}(-1)^{l-1}\, i_{\d}^{l-1} \tilde{f}^{j+2-2l}\right)  \\
   &+ \binom{j}{2} \binom{m-2}{m-j}^{-1} i_{\d}\left( \sum\limits_{l=1}^{[\frac{j}{2}]}(-1)^{l-1}\, i_{\d}^{l-1}  \tilde{f}^{j-2l}\right).
\end{align*}
Observe that $ \binom{m-j}{2}  \binom{m-2}{m-j-2}^{-1}= \binom{m-2}{m-j}^{-1} \binom{j}{2} = \binom{m}{2}\binom{m}{j}^{-1}$. This implies
\begin{align*}
     \binom{m}{2} (i_{\delta}v)_{i_1\cdots i_{j}1\cdots 1}=& \binom{m}{2}\binom{m}{j}^{-1} \left( \sum\limits_{l=1}^{[\frac{j+2}{2}]}(-1)^{l-1}\, i_{\d}^{l-1} \tilde{f}^{j+2-2l}  \right)
     +  \binom{m}{2}\binom{m}{j}^{-1}\left( \sum\limits_{l=1}^{[\frac{j}{2}]}(-1)^{l-1}\,i_{\d}^{l}  \tilde{f}^{j-2l}\right)\\
     =& \binom{m}{2}\binom{m}{j}^{-1} \left( \sum\limits_{l=0}^{[\frac{j}{2}]}(-1)^{l}\, i_{\d}^{l} \tilde{f}^{j-2l}  \right)
     +  \binom{m}{2}\binom{m}{j}^{-1}\left( \sum\limits_{l=1}^{[\frac{j}{2}]}(-1)^{l-1}\, i_{\d}^{l}  \tilde{f}^{j-2l}\right)\\
     =&\binom{m}{2}\binom{m}{j}^{-1} \tilde{f}^{j}.
\end{align*}
    This implies 
    \begin{align}
         (i_{\delta}v)_{i_1\cdots i_{j}1\cdots 1}
         =f_{i_1\cdots i_{j}1\cdots 1}, \quad 0\le j\le m-2.
    \end{align}
    \end{proof}

\smallskip

\begin{remark}\label{highest_MRT}
Using Theorem \ref{mrt_coro} we can simplify the choices of the amplitudes $a_0$ and $b_0$ in \eqref{Step0-Eq0}. We can take
\begin{equation}\label{amplitudes_1}
\begin{aligned}
a_0(x) = h(z)(z-\overline{z})^{(2-n)/2} \quad \mbox{and}\quad 
\overline{b_0(x}) = g(\theta)(z-\overline{z})^{(2m-n)/2},
\end{aligned}
\end{equation}
where $g$ and $h$ are as before.
Using these choices we get
\begin{equation*}
\int_{\Rb} r^{m-1}e^{\lambda r}\left(\widehat{W}^{m-1}_{i_1\cdots i_{m-1}}(\lambda,r,\theta) \prod_{j=1}^{m-1}(e_1+\I e_r)_{i_{j}}\right) \D r = 0, \quad \mbox{for a.e. }\theta \in \Sb^{n-2},\quad \forall \lambda \in \Rb,
\end{equation*}
in place of \eqref{Step0-Eq4}. Then we use Theorem \ref{mrt_coro} along with the fact that $W^{m-1}$ is compactly supported in the $x_1$ variable to show that  $W^{m-1}=0$ in $\Omega$ and subsequently $A^{m-1} = \widetilde{A}^{m-1}$ in $\Omega$. Then using a similar iteration argument as before, one can obtain $A^k=\widetilde{A}^k$ in $\Omega$, for all $k=1,\dots,m$ and $q=\widetilde{q}$.
\end{remark}

\section{Momentum ray transforms}\label{Sec_MRT}
In this section we prove a uniqueness result for the generalized MRT of tensor fields. We first start with an injectivity result for this transform on the tangent bundle on $\Rb^n$  and use this result to analyze the  same transform on the unit sphere bundle on $\Rb^n$. The latter result  is required to prove the main result of this paper.

 Let
$ S^m=S^m(C^{\infty}_{c}(\mathbb{R}^n)) $ denote the space of smooth compactly supported symmetric $ m $-tensor fields in $ \R^n $. An element $ F\in \mathbf{S}^m=S^0\oplus  S^1\oplus\cdots\oplus S^m $ can be written uniquely as
\begin{equation}\label{definition_of_F_m}
\begin{aligned}
	F &=\sum_{p=0}^{m}f^{(p)}=f^{(0)}_{i_0}+f^{(1)}_{i_1}\,\D x^{i_1} + \cdots+f^{(p)}_{i_1\cdots i_p}\D x^{i_1}\cdots \D x^{i_p}+\cdots+f^{(m)}_{i_1\cdots i_m}\D x^{i_1}\cdots \D x^{i_m}\\
    &	=\left( f^{(0)}_{i_0},f^{(1)}_{i_1},\cdots, f^{(m)}_{i_1\cdots i_m} \right)
\end{aligned}
\end{equation} which can be viewed as sum of a function, a vector field and up to a symmetric $m$-tensor field with $ f^{(p)}\in S^p$ for each $0\leq p\leq m$. Initially, we define MRT of $ F\in \mathbf{S}^m $, for  all $ (x,\xi) \in \R^n \times \R^n\setminus \{0\}$ and for all integers $ k\ge0$  as follows: 
\begin{equation}
	\begin{aligned}\label{Eq4.1}
		I^{m,k}F(x,\xi) &=\sum_{p=0}^{m}I^k\!f^{(p)}(x,\xi)\\
		&= \int_{-\infty}^{\infty} t^k \left(f^{(0)}_{i_0}(x+t\xi
		)+ f^{(1)}_{i_1}(x+t\xi)\, \xi^{i_1}+\cdots+ f^{(m)}_{i_1\cdots i_m}(x+t\xi)\, \xi^{i_1}\,\cdots \xi^{i_m}\right)\D t.
	\end{aligned}
\end{equation}

\begin{remark}
	For $ F= f^{(m)}\in S^m $, \eqref{Eq4.1} reduces to 
	\begin{equation}\label{Eq4.2}
		I^k f^{m}(x,\xi) = \int_{-\infty}^{\infty} t^k  f^{(m)}_{i_1\cdots i_m}(x+t\xi)\, \xi^{i_1}\,\cdots \xi^{i_m}\, \D t.
	\end{equation}
	This coincides with the classical definition of MRT introduced by Sharafutdinov in \cite{Sharafutdinov_1986_momentum,Sharafutdinov_book} and later studied in greater detail in \cite{KMSS,KMSS_range}.
\end{remark}

\subsection{Adjoint of MRT}
The tensor fields $F$ we are interested in recovering in Theorem \ref{mainresult_1} belongs to $C^{\infty}(\bar{\Omega})$. Extending by $0$ outside $\Omega$, we have that $F\in L^{\infty}(\Rb^n)$ with compact support. 
Note that, MRT above in \eqref{Eq4.1}) was defined for  smooth compactly supported symmetric tensor fields.
We now extend the definition of MRT to compactly supported distributions. To do so, let us first introduce the notion of $L^2$ adjoint of MRT and study MRT on compactly supported distributions.

Let $ \mathcal{E}'(S^m) $ be the space of compactly supported symmetric $m$-tensor field distributions and $\Dc'(\Rn \times (\Rn\setminus\{0\})) $ be the space of distributions. 
We first introduce the adjoint of MRT for $ f\in \Ec'(S^m)$. Then using this we will define the notion of adjoint of MRT for $F\in \Ec'(\mathbf{S}^m)$.

\begin{definition}\label{defintion_of_mrt_for_f}
	The momentum ray transforms \[ I^k: \mathcal{E}'(S^m)\rightarrow \mathcal{D}'(\Rn \times (\Rn\setminus\{0\}))  \] are linear operators, for every non-negative integer $k$, defined as follows:
	\[ \langle I^kf^{(m)},\psi \rangle = \langle f^{(m)}, (I^k)^*\psi \rangle \quad \mbox{for }\quad  \psi \in C_c^{\infty}(\Rn \times (\Rn\setminus\{0\})).\]

	Here for all indices $1\le i_1,\cdots,i_m\le n$, the components of  $ (I^k)^* \psi$ are given by 
	\begin{equation}\label{def:adjoint}
		[(I^k)^* \psi]_{i_1\cdots i_m}  
		= \int_{\Rn}\int_{\mathbb{R}}t^k\, \, \xi^{i_1}\cdots \xi^{i_m}\, \psi(x-t\xi,\xi) \D t\,\D \xi.
	\end{equation}
	
\end{definition}

\begin{remark}
	For $ k=0 $ we get \[ [(I^0)^* \psi]_{i_1\cdots i_m} =  \int_{\Rn}\int_{\mathbb{R}} \, \xi^{i_1}\cdots \xi^{i_m}\, \psi(x-t\xi,\xi) \D t\,\D \xi, \] which is the same as \cite[Eq. 2.5.3]{Sharafutdinov_book}.
\end{remark} 

Now we are ready to define the MRT for $F\in \mathcal{E}'(\mathbf{S}^m)$. To avoid notational confusion, starting from here, throughout the rest of this section, we use the $\psi$  for test functions to define the adjoint of $I^kf$ for  $f\in \Ec'(S^m)$ and $\Psi$ to define the adjoint of $I^{m,k} F$ for  $F\in \Ec'(\mathbf{S}^m)$. 
\begin{definition}\label{defintion_of_mrt_for_F} The momentum ray transforms $ I^{m,k}:\mathcal{E}'(\mathbf{S}^m) \rightarrow \mathcal{D}'(\Rn \times \Rn\setminus\{0\}) $ are defined as:
	\begin{equation}\label{def:mrt_of_F}
		\begin{aligned}
			\langle I^{m,k}F,\Psi\rangle &= \langle F,(I^{m,k})^*\Psi\rangle = \sum\limits_{p=0}^{m}\langle f^{(p)}, \lb I^{m,k}\rb^{*}_{p} \Psi\rangle \quad \Psi \in C_c^{\infty}(\Rn\times \Rn\setminus\{0\}), 
		\end{aligned}
	\end{equation}
	where  $ (I^{m,k})^*\Psi  $ is given by
	\begin{align}\label{adjoint_expression}
		((I^{m,k})^*\Psi)(x)= \lb \int_{\Rn}\int_{\mathbb{R}}t^k\, \, \Psi(x-t\xi,\xi) \D t\,\D \xi,\cdots, \int_{\Rn}\int_{\mathbb{R}}t^k\, \, \xi^{i_1}\cdots \xi^{i_m}\, \Psi(x-t\xi,\xi) \D t\,\D \xi\rb,
	\end{align}
    	which is an element of  $\oplus_{p=0}^{m} C^{\infty}(S^p)$ and $(I^{m,k})^{*}_{p} \Psi= \int_{\Rn}\int_{\mathbb{R}}t^k\, \, \xi^{i_1}\cdots \xi^{i_p}\, \Psi(x-t\xi,\xi) \D t\,\D \xi $.
    	
\end{definition}


\subsection{Injectivity result for MRT on the tangent bundle}
Now we state an injectivity result for MRT.
	\begin{theorem}\label{uniquenes_result_mrt}
		Suppose $ F \in \Ec'(\mathbf{S}^m)$. If  
		\begin{align*}
		     I^{m,k}F=0  \quad \mbox{for}\quad  k=0,1,\cdots, m  \quad \mbox{then}\quad F=0.
		\end{align*}
	\end{theorem}
\begin{remark}\label{mrt_uniq_for_single_tensor}
   We note that if $f^{(m)}\in \Ec'(S^m)  $ and  $I^kf^{(m)}=0$ for $k=0,1,\cdots,m$, then as a trivial consequence of Theorem \ref{uniquenes_result_mrt}, we obtain $f^{(m)}=0$.
   \end{remark}
	The proof of this result will be presented after the following lemma.  
\begin{lemma}\label{derivative_of_IkF}
Let $F\in \Ec'(\mathbf{S}^m)$.  Fix an index $i_s$ with $1\leq i_s\leq n$. Then for  any $k$, $0\le k\le m-1$,  the following identity holds
	\begin{equation}\label{key_identity_F_m}
		I^{m-1,k}(F)_{i_s}=  \PD_{\xi^{i_s}} I^{m,k}F- \PD_{x^{i_s}}I^{m,k+1}F
	\end{equation}
	in the sense of distributions. Here $ (F)_{i_s}= \sum_{p=0}^{m-1}(p+1)\, f^{(p+1)}_{i_1\cdots i_{p}i_s}\,  \D x^{i_1} \cdots \D x^{i_{p}} \in \Ec'(\mathbf{S}^{m-1})$.
\end{lemma}
\begin{proof}
	
	For $ f^{(p)}\in \Ec'(S^p) $, we have 
	\begin{align*}
		\langle\PD_{\xi^{i_s}} I^kf^{(p)},\psi \rangle &= - \langle I^kf^{(p)},\PD_{\xi^{i_s}}\psi\rangle, \qquad \psi \in C_c^{\infty}(\Rn\times \Rn\setminus\{0\}) \\
		&= - \langle  f^{(p)},(I^k)^* \left(\PD_{\xi^{i_s}}\psi\right)\rangle.
	\end{align*}
Note that \begin{align*}
(I^k)^* \left(\PD_{\xi^{i_s}}\psi\right)
=& \int_{\Rn}\int_{\mathbb{R}} t^k\, \xi^{i_1}\cdots \xi^{i_p}\, (\PD_{\xi^{i_s}}\psi)(x-t\xi,\xi) \D t\,\D \xi\\
	=& \int_{\Rn}\int_{\mathbb{R}} \,t^k \xi^{i_1}\cdots \xi^{i_p}\, \PD_{\xi^{i_s}} \left(\psi(x-t\xi,\xi) \right)\D t \,\D \xi \\&+ \int_{\Rn}\int_{\mathbb{R}} t^{k+1}\, \xi^{i_1}\cdots \xi^{i_p}\, (\PD_{x^{i_s}}\psi)(x-t\xi,\xi) \D t\,\D \xi\\
=&  \int_{\Rn} \PD_{\xi^{i_s}} \left(\int_{\mathbb{R}} t^k\, \xi^{i_1}\cdots \xi^{i_p}\,  \left(\psi(x-t\xi,\xi) \right)\D t\right)\,\D \xi \\
&-\sum_{j=1}^{p}\int_{\Rn}\int_{\mathbb{R}} t^k\, \,\xi^{i_1}\cdots \xi^{i_{j-1}}\,\delta_{i_s}^{i_j}\,\xi^{i_{j+1}}\,\cdots \xi^{i_{p}}\, \psi(x-t\xi,\xi) \D t\,\D \xi\\&+ \int_{\Rn}\int_{\mathbb{R}} t^{k+1}\,\xi^{i_1}\cdots \xi^{i_p}\, (\PD_{x^{i_s}}\psi)(x-t\xi,\xi) \D t\,\D \xi.
	\end{align*}
Since $\psi\in C_{c}^{\infty}(\Rb^n \times \Rb^n\setminus \{0\})$, we have 
\[
\int_{\Rn} \PD_{\xi^{i_s}} \left(\int_{\mathbb{R}}t^k\, \xi^{i_1}\cdots \xi^{i_p}\,  \left(\psi(x-t\xi,\xi) \right)\D t\right)\,\D \xi=0.
\] This finally gives
	\begin{align*}
		(I^k)^* \left(\PD_{\xi^{i_s}}\psi\right)=& \int_{\Rn}\int_{\mathbb{R}} t^{k+1}\,\xi^{i_1}\cdots \xi^{i_p}\, \xi^{i}\, (\PD_{x^{i_s}}\psi)(x-t\xi,\xi) \D t\,\D \xi\\&-  \sum_{j=1}^{p}\int_{\Rn}\int_{\mathbb{R}} t^k\, \,\xi^{i_1}\cdots \xi^{i_{j-1}}\,\delta_{{i_s}}^{i_j}\,\xi^{i_{j+1}}\,\cdots \xi^{i_{p}} \left(\psi(x-t\xi,\xi) \right)\D t\,\D \xi.
	\end{align*}
	Now
	\begin{align*}
		\langle\PD_{\xi^{i_s}} I^kf^{(p)},\psi\rangle 
		&= - \langle f^{(p)},(I^k)^* \left(\PD_{\xi^{i_s}}\psi\right)\rangle\\
		&= p\,\langle (f^{(p)})_{i_s}, (I^k)^*\psi \rangle - \langle f^{(p)},(I^{k+1})^* \left(\PD_{x^{i_s}}\psi\right)\rangle\\
		&= p\,\langle I^{k}(f^{(p)})_{i_s}, \psi\rangle - \langle I^{k+1}f^{(p)}, \left(\PD_{x^{i_s}}\psi\right)\rangle\\
		&= p\,\langle I^{k}(f^{(p)})_{i_s}, \psi\rangle +\langle \PD_{x^{i_s}}I^{k+1}f^{(p)}, \psi\rangle,
	\end{align*}
	where $ (f^{(p)})_{i_s}= f^{(p)}_{i_1\cdots i_{p-1}i_s}$ is a symmetric $ (p-1) $ tensor field.
	Therefore,
	\begin{equation}\label{key_identity}
		\begin{aligned}
			p\, \langle I^{k}(f^{(p)})_{i_s}, \psi\rangle=  \langle\PD_{\xi^{i_s}} I^kf^{(p)},\psi\rangle-\langle \PD_{x^{i_s}}I^{k+1}f^{(p)}, \psi\rangle,
		\end{aligned} 
	\end{equation}
	
	Thus for $ F\in \mathcal{E}'(\mathbf{S}^m) $, \eqref{key_identity} becomes
	\begin{equation*}
    		\langle I^{m-1,k}(F)_{i_s}, \Psi\rangle=  \langle\PD_{\xi^{i_s}} I^{m,k}F,\Psi\rangle-\langle \PD_{x^{i_s}}I^{m,k+1}F, \Psi \rangle.
	\end{equation*}
	This completes the proof.  
\end{proof}


	\begin{proof}[Proof of Theorem \ref{uniquenes_result_mrt}]
		The proof is based on induction in $m$. Note that here we are not using any inversion formula to prove the injectivity result.   For $m=0$, $F$
		is a function and $I^0F=0$ implies $F=0$; see \cite[Theorem 2.5.1]{Sharafutdinov_book}. Assume that the statement of  Theorem \ref{uniquenes_result_mrt} is true for some $m$. Then for $m+1$, from Lemma \ref{derivative_of_IkF} we obtain
		\begin{align*}
			I^{m,k}(F)_{i_{m+1}}=  \PD_{\xi^{i_{m+1} }} I^{m+1,k}F- \PD_{x^{i_{m+1}}}I^{m+1,k+1}F,
		\end{align*}
		for a fixed $i_{m+1}$ with $1\le i_{m+1} \le n $ and $0\le k\le m$. 
		By induction hypothesis now we have $I^{m,k}(F)_{i_{m+1}}=0 $ for $0\le k \le m$, which gives $(F)_{i_{m+1}}=0$. 
		Now varying $i_{m+1}$ we get $F=0$. This completes the proof. 
	\end{proof}
	
	Our next result says  that vanishing of only the $m^{\mathrm{th}}$ order generalized MRT implies injectivity. For an application of this theorem, see Remark \ref{highest_MRT}. We start with a lemma.

\begin{lemma}\label{translation_lemma}
Let $F \in \mathcal{E}'(\mathbf{S}^m)$, then the operator $ I^{m,k}F$ satisfies the  following relation
	\begin{align*}
		\langle\xi,\frac{\PD}{\PD x} \rangle^{p}  I^{m,k}F=  \begin{cases}
			(-1)^p\,\binom{k}{p}\,p!\, I^{m,k-p}F \quad \mbox{if} \quad p\le k\\
			0 \hspace{3.8cm} \mbox{if} \quad p>k.
		\end{cases} 
	\end{align*}

\end{lemma}
\begin{proof}
	By definition, for any test function $  \Psi \in C_c^{\infty}(\Rn \times \Rn\setminus\{0\})$, we have 
	\begin{align}\label{Eq4.13}
		\left\langle\langle\xi,\frac{\PD}{\PD x} \rangle^{p}I^{m,k}F,\Psi \right\rangle =
		 (-1)^p \left\langle F, (I^{m,k})^*\langle\xi,\frac{\PD}{\PD x} \rangle^{p}\Psi \right\rangle.
	\end{align}
	From  \eqref{adjoint_expression} we have
	\begin{align}\label{Eq4.14}
		(I^{m,k})^*\langle\xi,\frac{\PD}{\PD x} \rangle^{p}\Psi &= \sum_{l=0}^{m}\int_{\Rn}\int_{\mathbb{R}}t^k\, \, \xi^{i_1}\cdots \xi^{i_l}\, \langle\xi,\frac{\PD}{\PD x} \rangle^{p} \Psi(x-t\xi,\xi) \D t\,\D \xi\nonumber\\
		&=(-1)^p\sum_{l=0}^{m}\int_{\Rn}\int_{\mathbb{R}}t^k\, \, \xi^{i_1}\cdots \xi^{i_l}\,\frac{\D^p}{\D t^p}  \lb \Psi(x-t\xi,\xi)\rb  \D t\,\D \xi\nonumber\\
		&= \begin{cases}
			\frac{k!}{(k-p)!}\sum\limits_{l=0}^{m}\int_{\Rn}\int_{\mathbb{R}}t^{k-p}\, \, \xi^{i_1}\cdots \xi^{i_l}\,  \Psi(x-t\xi,\xi) \D t\,\D \xi \quad \mbox{if} \quad p\le k\nonumber\\
			0 \hspace{8.4cm}\mbox{if} \quad p> k
		\end{cases}\nonumber \\
		& = \begin{cases}
			\frac{k!}{(k-p)!}(I^{m,k-p})^*\Psi \quad \mbox{if} \quad p\le k\\
			0 \hspace{3.1cm}\mbox{if} \quad p> k.
		\end{cases}
	\end{align}
	In the third equality above we have used the integration by parts with respect to $ t $ variable. Combining \eqref{Eq4.13} and \eqref{Eq4.14} we obtain
	
	\begin{equation*}
		\begin{aligned}
			\langle\langle\xi,\frac{\PD}{\PD x} \rangle^{p}I^{m,k}F,\Psi \rangle= \begin{cases}
				(-1)^p\,\, \binom{k}{p}\,\,p!\,\,\langle I^{m,k-p} F,\Psi\rangle\quad \mbox{if} \quad p\le k\\
				0 \hspace{4.8cm}\mbox{if} \quad p> k.
			\end{cases}
		\end{aligned}
	\end{equation*}
	This completes the proof of the Lemma \ref{translation_lemma}.
\end{proof}

\begin{theorem}\label{mrt_coro_full}
	    Let $F \in \mathcal{E}'(\mathbf{S}^m) $ and $ I^{m,m} F(x,\xi) =0 $ for all $ (x,\xi)\in \R^n \times \R^n\setminus\{0\} $, then we have $ F=0 $.
	\end{theorem}
		
	\bpr The proof follows from Lemma \ref{translation_lemma} and Theorem \ref{uniquenes_result_mrt}.
	\epr
	
	In fact, one can prove a much stronger result in the form of an inversion algorithm for the generalized MRT as the following theorem shows. A trivial consequence of it is a partial data injectivity result; see Remark \ref{partial_mrt} . While we do not require the inversion algorithm to prove the main result of this paper, the result presented here could be of independent interest.
	
	\subsection{Inversion of MRT}\label{mrt_inversion}
To state our inversion formula we first introduce a notation.
For every non-negative integer $k$, the differential operator  $ P_{k}(\xi,\PD_{\xi}) $ of order $k$ is given by 
\begin{align}\label{definition_of_xidelxi_m}
  P_{k}(\xi,\PD_{\xi}) = \xi^{i_1} \cdots \xi^{i_k} \frac{\PD^k}{ \PD \xi^{i_1} \cdots \PD \xi^{i_k}}.
\end{align}

	\begin{theorem}\label{Th1.1}
		Suppose $ F $ is as in \eqref{definition_of_F_m}, then for every integer $j$ with $ 0\le j\le m $,  the following formula 
		\begin{equation}\label{Eq1.3}
			I^0f^{(j)}_{i_1\dots i_{j}}
			=(-1)^{m-j}\frac{\sigma(i_1\dots i_{j})}{(m-j)!\,(j)!}  P_{m-j}(\xi,\PD_{\xi})
			\sum_{k=0}^{j}\,(-1)^k \binom{j}{k} \frac{\partial^{j}(I^{m,k}\!F)}{\partial x^{i_1}\dots\partial x^{i_k}\partial\xi^{i_{k+1}}\dots\partial\xi^{i_{j}}},
		\end{equation}
		holds true for all indices $(i_1,\dots,i_{j})$. The RHS of \eqref{Eq1.3} is defined in the sense of distributions.
	\end{theorem}
To give the proof of Theorem \ref{Th1.1} we study few more properties of MRT, which will be presented in several lemmas. First, we recall the notion of  homogeneous distributions.
\begin{definition}\label{action_of_taul_smfn}
    	Let $\psi\in  C_c^{\infty}(\Rn\times \Rn\setminus\{0\})$ and  $0<1-\epsilon <\lambda<1+\epsilon $, where $\epsilon >0$ is  small. Define $ \tau_{\lambda}:  C_c^{\infty}(\Rn\times \Rn\setminus\{0\}) \rightarrow  C_c^{\infty}(\Rn\times \Rn\setminus\{0\}) $  by \[\tau_{\lambda}\psi(x,\xi) = \psi\left(x,\lambda\,\xi\right) \quad \mbox{ for all } \lambda \in (1-\epsilon,1+\epsilon).\]
\end{definition}
\vspace*{.4cm}
Next we define the action of $\tau_{\lambda}$ on distributions.
\begin{definition}\label{action_of_taul_distributions}
 For each $\lambda \in (1-\epsilon,1+\epsilon)$, define the distribution $\tau_{\lambda} I^kf^{(l)}$ by
 \begin{align*}
\langle \tau_{\lambda} I^kf^{(l)},\psi\rangle&= \lambda^{-n} 	\langle I^kf^{(l)},\tau_{\frac{1}{\lambda}} \psi\rangle \quad \psi \in C_c^{\infty}(\Rn\times \Rn\setminus\{0\})\nonumber\\
		&=\lambda^{-n}	\langle f^{(l)},(I^k)^*\tau_{\frac{1}{\lambda}} \psi\rangle .   
 \end{align*}
\end{definition}

\begin{lemma}[Homogeneity in the second variable]
	For every non-negative integer $k$, we have the following identities:
	\begin{align}\label{Eq4.10}
		\tau_{\lambda}	I^kf^{(l)}(x,\xi)=\lambda^{l-k-1} I^kf^{(l)}(x,\xi) \quad \mbox{for} \quad f^{(l)}\in \mathcal{E}'(S^l),
	\end{align}
	and for $F = \sum_{l=0}^{m}f^{(l)}_{i_1\cdots i_l}$,
	\begin{align}\label{homeogenity}
		\tau_{\lambda}	I^{m,k}F(x,\xi)= \sum_{l=0}^{m}\lambda^{l-k-1} I^kf^{(l)} \quad \mbox{for} \quad  F\in \mathcal{E}'(\mathbf{S}^m).
	\end{align}
\end{lemma}

\begin{proof}
	For any $ f \in \mathcal{E}'(S^l)  $ from Definition \ref{action_of_taul_distributions}, we have
	\begin{align}\label{Eq4.12}
		\langle \tau_{\lambda} I^kf^{(l)},\psi\rangle&= \lambda^{-n} 	\langle I^kf^{(l)},\tau_{\frac{1}{\lambda}} \psi\rangle \quad \psi \in C_c^{\infty}(\Rn\times \Rn\setminus\{0\})\nonumber\\
		&=\lambda^{-n}	\langle f^{(l)},(I^k)^*\tau_{\frac{1}{\lambda}} \psi\rangle.
	\end{align}	
	From \eqref{def:adjoint} we have 
	\begin{align*}
		(I^k)^* \tau_{\frac{1}{\lambda}}\psi  &=  \int_{\Rn}\int_{\mathbb{R}}t^k\, \, \xi^{i_1}\cdots \xi^{i_l}\, \tau_{\frac{1}{\lambda}}\psi(x-t\xi,\xi) \D t\,\D \xi \\
		&=\int_{\Rn}\int_{\mathbb{R}}t^k\, \, \xi^{i_1}\cdots \xi^{i_l}\, \psi(x-t\xi,\frac{\xi}{\lambda}) \D t\,\D \xi \\
		&= \lambda^{l+n} \int_{\Rn}\int_{\mathbb{R}}t^k\, \, \xi^{i_1}\cdots \xi^{i_l}\, \psi(x-t\lambda\xi,\xi) \D t\,\D \xi \quad [\xi \mapsto \lambda \xi]\\
		&= \lambda^{l+n-k-1} \int_{\Rn}\int_{\mathbb{R}}t^k\, \, \xi^{i_1}\cdots \xi^{i_l}\, \psi(x-t\xi,\xi) \D t\,\D \xi \quad [t \mapsto \lambda\, t]\\
		&= \lambda^{l+n-k-1}	(I^k)^* \psi.
	\end{align*}
	This together with \eqref{Eq4.12} gives
	\[ \langle \tau_{\lambda} I^kf^{(l)},\psi\rangle =\lambda^{l-k-1} \langle I^kf^{(l)},\psi\rangle.\]
	Thus for $F\in \mathcal{E}'(\mathbf{S}^m)$ with $F = \sum_{l=0}^{m}f^{(l)}_{i_1\cdots i_l}$ we have
	\[  \langle\tau_{\lambda} I^kF,\Psi\rangle= \langle\sum_{l=0}^{m} \lambda^{l-k-1}\,I^kf^{(l)},\Psi\rangle. \]
	This finishes the proof.
	\end{proof}

	\begin{lemma}
	Suppose $f^{(l)}\in \mathcal{E}'(S^l)$, then the following relations hold true. 
	\begin{align}
	    \frac{d^m}{d\lambda^m} \tau_{\lambda}	I^kf^{(l)}(x,\xi)\Big|_{\lambda=1}& = (l-k-1)(l-k-2)\cdots (l-k-m)\, I^kf^{(l)} \label{lambda_derivative}\\
	     P_{m}(\xi,\PD_{\xi})  I^kf^{(l)} &=
	      (l-k-1)(l-k-2)\cdots (l-k-m)\, I^kf^{(l)}. \label{xi_derivative}
	\end{align}
This implies 
\begin{align}\label{homogeneity_and_xi_derivative}
   \frac{d^m}{d\lambda^m} \tau_{\lambda}	I^kf^{(l)}(x,\xi)\Big|_{\lambda=1} = P_{m}(\xi,\PD_{\xi})  I^kf^{(l)}= \prod_{j=1}^{m} (l-k-j) \, I^kf^{(l)}. 
\end{align}

	\end{lemma}

\begin{proof}
    For $f^{(l)}\in \mathcal{E}'(S^l)$ from \eqref{Eq4.10}, we get
	\begin{align}\label{derivative_wrt_lambda}
	\frac{d^m}{d \lambda^m} \tau_{\lambda} I^kf^{(l)}\Big|_{\lambda=1} =(l-k-1) \cdots(l-k-m)\, I^kf^{(l)}. 
	\end{align}
	 This completes the proof of  \eqref{lambda_derivative}.\smallskip
	 
	To prove \eqref{xi_derivative}, we use an induction argument in $m$.
	For $m=1$ from \eqref{key_identity} 
	we have
	\begin{align*}
	    	l\, I^{k}(f^{(l)})_s=  \PD_{\xi^s} I^kf^{(l)}- \PD_{x^s}I^{k+1}f^{(l)}.
	\end{align*}
	Now multiplying both sides of above by $\xi^s$ and taking summation over $s$, we obtain
	\begin{align*}
	    	l\, I^{k}f^{(l)}=   \xi^s \PD_{\xi^s} I^kf^{(l)}- \xi^s\PD_{x^s}I^{k+1}f^{(l)}.
	\end{align*}
This together with Lemma \ref{translation_lemma}    gives
	\begin{align}\label{derivative_wrt_xi}
	 P_{1}(\xi,\PD_{\xi}) I^kf^{(l)} =   \langle \xi,\PD_{\xi}\rangle I^k f^{(l)} &= (l-k-1) I^kf^{(l)}.
	\end{align}
Assume \eqref{xi_derivative} is true for some $m$. This implies
\begin{align}\label{induction_hypo}
    P_m(\xi,\PD_{\xi})=(l-k-1)(l-k-2)\cdots (l-k-m)\, I^kf^{(l)}.
\end{align}
Now apply $ P_{1}(\xi,\PD_{\xi}) $ in above and obtain
\begin{align}\label{induction_step_m}
   P_{1}(\xi,\PD_{\xi})\left(  P_{m}(\xi,\PD_{\xi})I^k f^{(l)}\right) &=   (l-k-1)(l-k-2)\cdots (l-k-m)\,  P_{1}(\xi,\PD_{\xi})\,I^kf^{(l)}.
   \end{align}
   A simple calculation gives 
   \begin{align}\label{a_product_relation}
        P_{1}(\xi,\PD_{\xi}) P_{m}(\xi,\PD_{\xi}) &= P_{m+1}(\xi,\PD_{\xi}) + m P_{m}(\xi,\PD_{\xi}). 
   \end{align}
   Combining \eqref{derivative_wrt_xi},\eqref{induction_step_m} and \eqref{a_product_relation} we obtain
    \begin{align*}
    P_{m+1}(\xi,\PD_{\xi})  I^k f^{(l)})+ m  P_{m}(\xi,\PD_{\xi})\,I^kf^{(l)} &= (l-k-1)^2(l-k-2)\cdots (l-k-m)\, I^kf^{(l)}.
    \end{align*}
    This together with \eqref{induction_hypo} gives
    \begin{align*}
   P_{m+1}(\xi,\PD_{\xi})\,  I^k f^{(l)}= (l-k-1)(l-k-2)\cdots (l-k-m) (l-k-1-m) \, I^kf^{(l)}.
\end{align*}
This completes the proof of \eqref{xi_derivative}. 
\end{proof}

\begin{lemma}\label{Lm1.2}
		Suppose $F = \sum_{l=0}^{m}f^{(l)}_{i_1\cdots i_l} \in \mathcal{E}'(\mathbf{S^m})$, then we have the following identity: 
		\[   I^0 f^{0}(x,\xi)  = \frac{(-1)^m}{m!} P_{m}(\xi,\PD_{\xi}) I^{m,0}F. \]
	\end{lemma}
	 \begin{proof}
	    Let $\epsilon$ be a positive number such that $0<1-\epsilon <\lambda < 1+\epsilon$, then substituting $k=0$ in \eqref{homeogenity}  we obtain
	    \begin{align*}
	        \tau_{\lambda}	I^{m,0}F(x,\xi)= \sum_{p=0}^{m}\lambda^{p-1} I^0f^{(p)}.
	    \end{align*}
	    Differentiating this $m$ times with respect to $\lambda$ we get
	    \begin{align*}
	        \frac{d ^m}{d \lambda^m} \tau_{\lambda}	I^{m,0}F =\frac{d ^m}{d \lambda^m}( \lambda^{-1})\, I^0f^{(0)}= \frac{(-1)^m\, m!}{\lambda^{m+1}}\, I^0f^{(0)}.
	    \end{align*}
	    Evaluating above at $\lambda=1$ gives
	    \[ I^0f^{(0)}= \frac{(-1)^m}{m!}\,\frac{d ^m}{d \lambda^m} \tau_{\lambda}	I^{m,0}F \Big|_{\lambda=1}. \]
	    This together with \eqref{homogeneity_and_xi_derivative} implies
	    \begin{equation*}
	      I^0f^{(0)}= \frac{(-1)^m}{m!}\, P_{m}(\xi,\PD_{\xi})	I^{m,0}F.  
	    \end{equation*}
	 \end{proof}
	\begin{proof}[Proof of Theorem \ref{Th1.1}] We prove the result by induction on $m$.
	For $m=0$ \eqref{Eq1.3} holds trivially. Now for any $m\geq 1$ and $ j=0, $ \eqref{Eq1.3} reduces to 	\[ I^0f^{(0)}= \frac{(-1)^m}{m!}P_{m}(\xi,\PD_{\xi}) I^{m,0}F,\] and this is proved in Lemma \ref{Lm1.2}. 
			   Assume \eqref{Eq1.3} to be true for some $m\geq 1$ and all $j$ with $ 0\le j\le m $. 
		For $ m+1 $, fix an index $i_{j+1}\, (1\le i_{j+1}\le n)$.
		

		Then  from Lemma \ref{derivative_of_IkF} we obtain
		\begin{align*}
			&I^{m,k}(F)_{i_{j+1}}=\frac{\PD I^{m+1,k}F}{\PD \xi_{i_{j+1}}}-\frac{\PD I^{m+1,k+1}F}{\PD x_{i_{j+1}}} \quad \mbox{for} \quad 0\le k \le m.
		\end{align*}
Here $(F)_{i_{j+1}} \in \mathcal{E}'(\mathbf{S}^m)$ is given by	
\[ 	(F)_{i_{j+1}}= \sum_{p=0}^{m}(p+1)\,f^{(p+1)}_{i_1\cdots i_{p}i_{j+1}}\,  \D x^{i_1} \cdots \D x^{i_{p}}=  \sum_{p=0}^{m}\left(\wt{f_{i_{j+1}}^{(p)}}\right)_{i_1\cdots i_{p}}\,  \D x^{i_1} \cdots \D x^{i_{p}},\]
	where $\left(\wt{f_{i_{j+1}}^{(p)}}\right)_{i_1\cdots i_{p}} =  (p+1)\,f^{(p+1)}_{i_1\cdots i_{p}i_{j+1}}$.	 
		Now using induction hypothesis we have that
		\begin{align*}
			I^0\left(\wt{f_{i_{j+1}}^{(j)}}\right)_{i_1\cdots i_{j}}&=(-1)^{m-j}\frac{\sigma(i_1\dots i_{j})}{(m-j)!\,(j)!} P_{m-j}(\xi,\PD_{\xi}) \sum\limits_{k=0}^{j}\,(-1)^k \binom{j}{k}\frac{\partial^{j} I^{m,k}(F)_{i_{j+1}} }
			{\partial x_{i_1}\dots\partial x_{i_k}\partial\xi_{i_{k+1}}\dots\partial\xi_{i_{j}}}.
		\end{align*}
		We replace $ \sigma(i_1\dots i_{j}) $ by the stronger operator $\sigma(i_1\dots i_{j+1})$ because the left hand side is symmetric with respect to all indices. Therefore we have
		\begin{equation}\label{Eq1.6}
			\begin{aligned}
				I^0f^{(j+1)}_{i_1\cdots i_{j}i_{j+1}}&=	\frac{(-1)^{m-j}\sigma(i_1\dots i_{j+1})}{(m-j)!\,(j+1)!}P_{m-j}(\xi,\PD_{\xi})\\&\bigg[  \sum\limits_{k=0}^{j}\,(-1)^k \binom{j}{k}\frac{\partial^{j+1}{I^{m+1,k}F}}
				{\partial x_{i_1}\dots\partial x_{i_k}\partial\xi_{i_{k+1}}\dots\partial\xi_{i_{j+1}}}- \sum\limits_{k=0}^{j}\,(-1)^k \binom{j}{k}\frac{\partial^{j+1}{I^{m+1,k+1}F}}
				{\partial x_{i_1}\dots\partial x_{i_k} x_{i_{j+1}}\partial\xi_{i_{k+1}}\dots\partial\xi_{i_{j}}}
				\bigg].
			\end{aligned}
		\end{equation}
		Using the arguments used in \cite[Theorem 3.1]{KMSS} we can write
		\begin{align*}
			&\bigg[  \sum\limits_{k=0}^{j}\,(-1)^k \binom{j}{k}\frac{\partial^{j+1}{I^{m+1,k}F}}
			{\partial x_{i_1}\dots\partial x_{i_k}\partial\xi_{i_{k+1}}\dots\partial\xi_{i_{j+1}}}- \sum\limits_{k=0}^{j}\,(-1)^k \binom{j}{k}\frac{\partial^{j+1}{I^{k+1}F}}
			{\partial x_{i_1}\dots\partial x_{i_k} x_{i_{j+1}}\partial\xi_{i_{k+1}}\dots\partial\xi_{i_{j}}}
			\bigg] \\& = \sum\limits_{k=0}^{j+1}\,(-1)^k \binom{j+1}{k}\frac{\partial^{j+1}(I^{m+1,k+1}\!F)}
			{\partial x^{i_1}\dots\partial x^{i_k}\partial\xi^{i_{k+1}}\dots\partial\xi^{i_{j+1}}}. 
		\end{align*}
		Now from \eqref{Eq1.6} we obtain 
		\begin{align*}
			&I^0f^{(j+1)}_{i_1\dots i_{j} i_{j+1}}\\&= \frac{(-1)^{m+1-j-1}\sigma(i_1\dots i_{j+1})}{(m+1-j-1)!\,(j+1)!}P_{m+1-j-1}(\xi,\PD_{\xi})\sum\limits_{k=0}^{j+1}\,(-1)^k \binom{j+1}{k}\frac{\partial^{j+1}(I^{m+1,k+1}\!F)}
			{\partial x^{i_1}\dots\partial x^{i_k}\partial\xi^{i_{k+1}}\dots\partial\xi^{i_{j+1}}}.
		\end{align*}
		Thus using induction hypothesis we prove \eqref{Eq1.3} for any  $ m $ and all $j$ with $ 1\le j\le m $  with $F\in \mathcal{E}'(\mathbf{S}^m) $. 
		This completes the proof of Theorem \ref{Th1.1}.
	\end{proof}

	\begin{remark}\label{partial_mrt}
	    We have shown that if $F\in \mathcal{E}'(\mathbf{S}^m)$ is such that $I^{m,l}F=0$ for $l=0,1,\cdots,k$ with $k<m$, then Theorem \ref{Th1.1} implies $ f^{(j)}=0 $ for $j=0,1,\cdots,k$. It is interesting to study the kernel description of operator $F \mapsto (I^{m,0}F,\cdots,I^{m,k}F)$ for $k<m$ similar to \cite{Rohit_Suman}. 
	\end{remark}

\subsection{MRT on the unit sphere bundle of $\Rb^n$}
Recall that in Section \ref{Section_coeff_deter}, we have MRT of a sum of different rank tensor fields only over the unit sphere bundle. Therefore in this section, we analyze such MRT. Due to this restriction, this transform will have a non-trivial kernel as we show below, in contrast to the situation in Theorem \ref{uniquenes_result_mrt}.

Let $J^{m,k}F(x,\xi)= I^{m,k}F(x,\xi)\big|_{\Rn \times \Sn}$. We study the operator  
$F \mapsto (J^{m,0}F,\cdots,J^{m,m}F)$ and show that it has a non-trivial kernel for $m\ge 2$.
Recall the operator $i_{\d}:S^{m} \rightarrow S^{m+2}$,
\begin{align}\label{def_i_d}
(i_{\d}f)^{(m+2)}= (i_{\d} f^{(m)})_{i_1\cdots i_{m+2}}=  \sigma(i_1\dots i_{m+2}) f^{(m)}_{i_1\dots i_{m}}\d_{i_{m+1}i_{m+2}}.
\end{align}
Similarly, $i^p_{\d}:S^{m} \rightarrow S^{m+2p}$ is defined by taking composition of $i_{\d}$ with itself $p$ times. 
\begin{proposition}\label{prop_1}
   Let $f^{(l)}\in S^l$. Then $i^p_{\d}f^{(l)}$ satisfies the following relation
   \begin{align}\label{increasing_order}
       J^k (i^p_{\d}f^{(l)})(x,\xi) = J^k f^{(l)}(x,\xi), \quad \mbox{for any integer} \quad p\ge 1.
   \end{align}
\end{proposition}
\bpr
This follows by straightforward computation keeping in mind that $\xi\in \Sb^{n-1}$.
\epr

\begin{theorem}\label{Th_MRT_unit_disk}
Let $m\ge 2$ and $F =\sum\limits_{l=0}^{m} f^{(l)} \in \Sm^m$. 
Then  
\[J^{m,k}F(x,\xi)=0  \quad \mbox{for} \quad k=0,1,\cdots,m,\]
if and only if 
\begin{align*}\label{even_highest_component}
      f^{(2[\frac{m}{2}])}
=- i_{\d}\left( \sum\limits_{l=1}^{[\frac{m}{2}]} i^{[\frac{m}{2}]-l}_{\d} f^{(2l-2)}\right),
    \end{align*}
and 
\begin{align*}
    f^{(2[\frac{m-1}{2}]+1)}&= -i_{\d}\left( \sum\limits_{l=1}^{[\frac{m-1}{2}]} i^{[\frac{m}{2}]-l}_{\d} f^{(2l-1)}\right).
\end{align*}
Moreover, $J^{m,k}F=0$ if and only if $F=0$ for $m=0,1$.
    \end{theorem}
\begin{proof}
First assume $m\geq 2$. We have that $J^{m,k}F(x,\xi)=0$ for all $(x,\xi)\in \Rn \times \Sn $. This implies $J^{m,k}F (x,\pm\xi)=0$. From \eqref{Eq4.1} we have 
\begin{align*}
    J^{m,k}F(x,-\xi)= \sum\limits_{l=0}^{m}(-1)^{l-k} J^{k}f^{(l)}(x,\xi) \quad \mbox{for} \quad k=0,1,\cdots,m.
\end{align*}
This implies for each $\xi \in \Sb^{n-1}$ we have
\begin{equation}\label{add_sub_relation}
\begin{aligned}
J^{m,k}F(x,\xi)\pm J^{m,k}F(x,-\xi) &=\sum\limits_{l=0}^{m}(1 \pm (-1)^{l-k}) J^{k}f^{(l)}(x,\xi)\quad \mbox{for} \quad k=0,1,\cdots,m.
\end{aligned}
\end{equation}
Since the left hand side of \eqref{add_sub_relation} is $0$ by assumption, 
we obtain
\begin{align}\label{separating_relation_of_mrt}
    \sum\limits_{l=0}^{[\frac{m}{2}]} J^{k}f^{(2l)}(x,\xi)= 0= \sum\limits_{l=0}^{[\frac{m-1}{2}]} J^{k}f^{(2l+1)}(x,\xi) \quad \mbox{for} \quad k=0,1,\cdots,m.
\end{align}
Let us now define the tensor fields
\begin{equation}
\begin{aligned}
     G_{1}:= \sum\limits_{l=0}^{[\frac{m}{2}]} i^{[\frac{m}{2}]-l}_{\d} f^{(2l)} \qquad \mbox{and}\qquad 
    G_{2}:= \sum\limits_{l=0}^{[\frac{m-1}{2}]} i^{[\frac{m-1}{2}]-l}_{\d} f^{(2l+1)}.
\end{aligned}
\end{equation}
Combining Proposition \ref{prop_1} and \eqref{separating_relation_of_mrt}, we obtain
\begin{align}\label{mrt_of_g_1_g_2}
    J^k G_1(x,\xi)=0 \quad \mbox{and} \quad J^k G_2(x,\xi)=0,
\end{align}
for $k=0,1,\cdots,m$ and $(x,\xi)\in \Rn\times 
\Sn$. 

Note that $G_1$ and $G_2$ are a symmetric tensor field of order at most $m$. From \cite[Equation 2.6]{KMSS} we get 
$J^kG_i=0$ implies  $I^{k} G_{i}=0$ for $i=1,2$ and for $0\leq k\leq m$. From Theorem \ref{uniquenes_result_mrt} (in fact, Remark \ref{mrt_uniq_for_single_tensor}) we have
$G_1=G_2=0$ in $\Rn$. 

This  implies 
\begin{align}\label{relation_for_g_1}
f^{(2[\frac{m}{2}])}
= -\sum\limits_{l=0}^{[\frac{m}{2}]-1} i^{[\frac{m}{2}]-l}_{\d} f^{(2l)} 
=- i_{\d}\left( \sum\limits_{l=0}^{[\frac{m}{2}]-1} i^{[\frac{m}{2}]-l-1}_{\d} f^{(2l)}\right)
=- i_{\d}\left( \sum\limits_{l=1}^{[\frac{m}{2}]} i^{[\frac{m}{2}]-l}_{\d} f^{(2l-2)}\right).
\end{align}
Similarly,
\begin{align}\label{relation_for_g_2}
     f^{(2[\frac{m-1}{2}]+1)}&= -i_{\d}\left( \sum\limits_{l=1}^{[\frac{m-1}{2}]} i^{[\frac{m}{2}]-l}_{\d} f^{(2l-1)}\right).
\end{align}
Thus we have shown that if $ J^{m,k}F=0$ for $k=0,1,\cdots,m$ then $f^{(2[\frac{m}{2}])}$ and $f^{(2[\frac{m-1}{2}]+1)}$ can be written as \eqref{relation_for_g_1} and \eqref{relation_for_g_2} respectively.

For the other implication, assume that $f^{(2[\frac{m}{2}])}$ and $f^{(2[\frac{m-1}{2}]+1)}$ can be written as above. Then the fact that $J^{m,k} F=0$ for $0\le k\le m$ follows from a direct computation using Proposition \ref{prop_1}. 

For $m=0,1$,  $F=0$ implies $ J^{m,k}F=0$.
 Suppose $m=1$ and $J^{1,k}F(x,\xi) = 0$, for $k=0,1$, where $F=f^{(0)}+ f^{(1)}_i \D x^i$. 
     Then using homogeneity in the second variable we have that $J^{1,k}f^{(0)}=J^{1,k}f^{(1)}=0 $ for $k=0,1$. From \cite[Equation 2.6]{KMSS} we get  $I^{1,k}f^{(0)}=I^{1,k}f^{(1)}=0 $ for $k=0,1$. Using Remark \ref{mrt_uniq_for_single_tensor} we have $f^{(0)}=f^{(1)}=0$. This implies $F=0$. The case $m=0$ is similar, in fact, simpler. 
    This completes the proof.
\end{proof}

Similar to Theorem \ref{mrt_coro_full}, for the case of unit sphere bundle we have the following result showing uniqueness using only the highest order moment.
\begin{theorem}\label{mrt_coro}
	    Let $F \in \mathcal{E}'(\mathbf{S}^m)$ and $J^{m,m} F(x,\xi) =0$ for all $(x,\xi)\in \R^n \times \Sb^{n-1}$, then we have the same conclusion as in Theorem \ref{Th_MRT_unit_disk}. 
	\end{theorem}
	\bpr 
	The proof follows from restricting Lemma \ref{translation_lemma} to the unit sphere bundle of $\Rb^n$ and then combining it with Theorem \ref{Th_MRT_unit_disk}.
	\epr

\begin{remark}
   We know that any symmetric $m$ tensor field in $\Rn$ has $\binom{m+n-1}{m}$ distinct components. Suppose $ F\in \mathcal{E}'(\mathbf{S}^m)$ as in \eqref{definition_of_F_m}. Then $F$ has a total of $ \sum\limits_{l=0}^m \binom{l+n-1}{l}= \binom{m+n}{m}$ distinct components. Therefore recovery of $F$ in $\mathbb{R}^n$ is, in effect, equivalent to recovering a symmetric  $m$ tensor field in  $\mathbb{R}^{n+1}$ from the knowledge of MRT. However to recover a symmetric $m$ tensor field in $\mathbb{R}^{n+1}$,  first $m+1$ MRT are  required. 
   In this section we have the first $m+1$ MRT of $F$ on a $2n-1$ dimensional data set. Thus based on this information alone, one cannot  recover the entire field $F$; see Theorem \ref{Th_MRT_unit_disk}.  However having MRT on a $2n$ dimensional set  leads to unique recovery of the entire $F$; see Theorem \ref{uniquenes_result_mrt}. 
   
 
\end{remark}


\appendix
\section{Semiclassical elliptic regularity}\label{Appendix_elliptic_regu}
Let $\Omega\subset \Rb^n$ be an open bounded domain and $\partial\Omega$ be smooth. Let $a_{jk}(x)$, $b_j(x)$ and $c(x)$ be smooth functions on $\overline{\Omega}$ such that $\langle a_{jk}\xi,\xi\rangle \geq \theta(x)|\xi|^2$ for all $\xi \in \Rb^n\setminus\{0\}$, where $\theta(x)>\ve_0>0$ in $\Omega$.
We define the following general second order elliptic operator
\[ L(x,D) := -\sum_{j,k=1}^{n} \frac{\partial}{\partial x^j} \left(a_{jk}(x)\frac{\partial}{\partial x^k}\right) + \sum_{j=1}^{n} b_{j}(x)\frac{\partial}{\partial x^j} + c(x),\qquad \mbox{in }\Omega.
\]
We prove the following semiclassical regularity estimate. Probably this result is well known but we could not find an easily available reference.
The proof given here follows the classical elliptic regularity estimate in \cite[Theorem 5, page 323]{Evans_pde_book} very closely.
\begin{proposition}\label{Prop_ellip_regu}
Let $L(x,D)$ be defined in $\Omega$ as above.
Let $0\leq j\leq 2m-2$ and  $0<h\ll 1$.  For $u \in H^{j+2}(\Omega) \cap H^1_0(\Omega)$, we have the following estimate
\begin{equation}\label{ellip_regu}
\| u \|_{H^{j+2}_{\mathrm{scl}}(\Omega)} 
\leq C \left(\|(-h^2 L  u\|_{H^{j}_{\mathrm{scl}}(\Omega)} + \|u\|_{L^2(\Omega)}\right).
\end{equation}
\end{proposition}

\begin{proof}
We recall the classical elliptic regularity estimate (\cite[Theorem 5]{Evans_pde_book},  \cite[Theorem 8.13]{Gilbarg_Trudinger} ) for $u \in H^{j+2}(\Omega)\cap H^1_0(\Omega)$:
\begin{equation}\label{ellip_regu_1}
	\|u\|_{H^{j+2}(\Omega)} \leq C\left(\|L u\|_{H^{j}(\Omega)} + \|u\|_{L^2(\Omega)}\right), \quad \mbox{for all}\quad  0\leq j\leq 2m-2.
\end{equation}
Using this, we prove estimate \eqref{ellip_regu} for the case $j=0$.
 From \eqref{ellip_regu_1}, we have, with $j=0$, 
\[
\begin{aligned}
\sum_{i,j=1}^{n} \left\| D_{ij} u\right\|_{L^2(\Omega)} 
\leq& C\left(\|L u\|_{L^2(\Omega)} + \|u\|_{L^2(\Omega)}\right),\\
\sum_{i=1}^{n} \left\| D_{i} u\right\|_{L^2(\Omega)} 
\leq& C\left(\|L u\|_{L^2(\Omega)}+\|u\|_{L^2(\Omega)}\right),
\end{aligned} 
\]
This implies
\begin{equation}\label{ellip_1}
    \sum_{i,j=1}^{n} \left\| h^2D_{ij} u\right\|_{L^2(\Omega)} \leq C\left(\|h^2 L (x,D)u\|_{L^2(\Omega)} + \|u\|_{L^2(\Omega)}\right) \mbox{ since } h\ll 1.
\end{equation}

On the other hand, using integration by parts, we see
\begin{equation}\label{ellip_2}\begin{aligned}
\sum_{i=1}^{n} \left\| hD_{i} u\right\|^2_{L^2(\Omega)}
=& \sum_{i=1}^{n}\left\langle hD_i u,hD_i u\right\rangle_{L^2(\Omega)}\\
=& \sum_{i=1}^{n}\left\langle h^2D^2_i u, u\right\rangle_{L^2(\Omega)}, \quad &&(\mbox{using }u|_{\partial\Omega}=0)\\
\leq& C \lb \sum\limits_{i=1}^{n}\lVert h^2 D_i^2 u\rVert^2 + \lVert u \rVert^2\rb \\
\leq& C\left(\|h^2 L u \|^2_{L^2(\Omega)} + \|u \|^2_{L^2(\Omega)}\right), \quad &&\mbox{(using \eqref{ellip_1})}.
\end{aligned}\end{equation}
Therefore, summing up \eqref{ellip_1} and \eqref{ellip_2} we have
\begin{equation}\label{step_0}
\|u\|_{H^2_{\mathrm{scl}}(\Omega)} \leq C\left(\|h^2 L u \|_{L^2(\Omega)} + \|u \|_{L^2(\Omega)}\right).
\end{equation}

Next, for $0<l<2m-2$,  assume by induction that 
\[	\|u\|_{H^{l+2}_{\mathrm{scl}}(\Omega)} 
	\leq C\left(\|h^2 L u \|_{H^{l}_{\mathrm{scl}}(\Omega)} + \|u \|_{L^2(\Omega)}\right),\quad \mbox{ for } u\in H^{l+2}(\Omega)\cap H^1_{0}(\Omega),
\]

We show that the estimate is true for $l+1$ with $u\in H^{l+3}(\O) \cap H^1_0(\O)$. 
Let $\A$ be a multi-index such that $|\A|=l+1$ and $\A_n=0$.
Denote $\wt{u}=h^{l+1}D^{\A}u$. Note that $\wt{u}$  satisfies $\wt{u}|_{\PD \O}=0$ since $\A_n=0$, and $\wt{u}\in H^2(\O)$, since $u\in H^{l+3}(\O)$.  We have 
\[\begin{gathered}
	h^2 L \wt{u} = \wt{f} \mbox{ in } \O, \mbox{ where }\\
	\wt{f} 
	= h^{l+3}D^{\A}\left(L u\right) 
	+\sum_{\substack{\B\leq\A\\ \B\neq\A}}h^{l+3} {\A \choose \B}\Bigg{\{}\sum\limits_{i,j=1}^{n}\lb (D^{\A-\B}a^{ij}) D^{\B}(D_ju)\rb_{x_{j}} - (D^{\A-\B}b^j)D^{\B}(D_ju)- (D^{\A-\B}c) D^{\B}u\Bigg{\}}.	
\end{gathered}\]

We have that 
\[
\lVert \wt{f}\rVert_{L^2(\O)}\leq C \lb \lVert h^2 L u\rVert_{H^{l+1}_{\mathrm{scl}}(\O)}+h\lVert u\rVert_{H^{l+2}_{\mathrm{scl}}(\O)}\rb\leq C\lb \lVert h^2 L u\rVert_{H^{l+1}_{\mathrm{scl}}(\O)}+\lVert u\rVert_{L^2(\O)}\rb.
\]
The last inequality follows from the induction assumption.

Since $\wt{u}\in H^2(\O) \cap H^1_0(\O)$, we have from \eqref{step_0} that 
\[
\lVert \wt{u}\rVert_{H^2_{\mathrm{scl}}(\O)} \leq C \lb \lVert \wt{f}\rVert_{L^2(\O)}+ \lVert \wt{u}\rVert_{L^2(\O)}\rb.
\]
Combining with the induction assumption, we get, 
\[
\lVert \wt{u}\rVert_{H^2_{\mathrm{scl}}(\O)}\leq C \lb \lVert h^2 L u\rVert_{H^{l+1}_{\mathrm{scl}}(\O)}+ \lVert u\rVert_{L^2(\O)}\rb.
\]
Therefore, for any $\B$ with $|\B|=l+3$ and $\B_n= 0,1 \mbox{ or } 2$, we have that 
\[
\lVert h^{|\B|}D^{\B} u\rVert_{L^2(\O)}\leq C \lb \lVert h^2 L u\rVert_{H^{l+1}_{\mathrm{scl}}(\O)}+ \lVert u\rVert_{L^2(\O)}\rb. 
\]
Now assume by induction that for any $\B$ with $|\B|=l+3$ and $\B_n=r>2$, we have that 
\Beq\label{second_induction_assumption}
\lVert h^{|\B|}D^{\B} u\rVert_{L^2(\O)}\leq C \lb \lVert h^2 L u\rVert_{H^{l+1}_{\mathrm{scl}}(\O)}+ \lVert u\rVert_{L^2(\O)}\rb. 
\Eeq
Let us consider a $\B$ with $|\B|=l+3$ and $\B_n=r+1$. We write $\B= \g+\delta$ where $\delta=(0,\cdots,0,2)$.

Then 
\[	h^{l+1}D^{\g} h^2 L u = a_{nn}h^{l+3}D^{\B} u + h^{l+3}\sum_{\substack{|\beta|=0\\ \B_n=0,1,\dots,r}}^{l+3}\left[C_{\B}D^{\B}u \right],
\]
for some smooth functions $C_{\B}$. 

Since $a_{nn}\geq \theta>0$ by ellipticity, we have, using the induction assumption \eqref{second_induction_assumption}, 
\[
\lVert h^{l+3} D^{\B} u\rVert_{L^2(\O)} \leq C \lb \lVert h^2 L u\rVert_{H^{l+1}_{\mathrm{scl}}(\O)}+\lVert u\rVert_{L^2(\O)}\rb. 
\]
Finally, we have, 
\[
\lVert u\rVert_{H^{l+3}_{\mathrm{scl}}(\O)} \leq C \lb \lVert h^2 L u\rVert_{H^{l+1}_{\mathrm{scl}}(\O)}+\lVert u\rVert_{L^2(\O)}\rb.
\]
The proof is now complete.
\epr

\section*{Acknowledgments}
This work was initiated during a visit of S.B. to TIFR Centre for Applicable Mathematics, Bangalore, India, which was partly supported by a SERB MATRICS research grant of V.P.K.
S.B. is partly supported by a SEED grant of Indian Institute of Science Education and Research, Bhopal, India. S.K.S. was partly supported by the European Research Council under Horizon 2020 (ERC CoG 770924).
\bibliographystyle{plain}
\bibliography{bibfile.bib}
\end{document}